\documentclass[12pt,reqno]{amsart}

\usepackage{fullpage}
\usepackage{comment}
\usepackage{enumerate}
\usepackage{amssymb}
\usepackage{amsthm}
\usepackage[normalem]{ulem}
\usepackage{thmtools}
\usepackage{mathtools}
\usepackage{pgf,tikz}
\usepackage{float}
\usetikzlibrary{arrows}
\usepackage{hyperref}
\usepackage{cite}
\usepackage{color}
\usepackage{transparent}
\usepackage{graphicx}
\usepackage{subcaption}

\definecolor{azure(colorwheel)}{rgb}{0.0, 0.5, 1.0}
\definecolor{hanpurple}{rgb}{0.32, 0.09, 0.98}
\definecolor{iris}{rgb}{0.35, 0.31, 0.81}
\definecolor{byzantine}{rgb}{0.74, 0.2, 0.64}
\definecolor{ashgrey}{rgb}{0.7, 0.75, 0.71}
\definecolor{battleshipgrey}{rgb}{0.52, 0.52, 0.51}

\usepackage{xcolor}
\usepackage[normalem]{ulem}

\hypersetup{colorlinks=true,pdfborder=0 0 0,  citecolor  = blue, citebordercolor= {magenta}}
\hypersetup{linkcolor=black}
\makeatletter
\let\reftagform@=\tagform@
\def\tagform@#1{\maketag@@@{(\ignorespaces\textcolor{purple}{#1}\unskip\@@italiccorr)}}
\renewcommand{\eqref}[1]{\textup{\reftagform@{\ref{#1}}}}
\makeatother

\makeatletter

\DeclareUrlCommand\ULurl@@{%
  \def\UrlLeft{\uline\bgroup}%
  \def\UrlRight{\egroup}}
\def\ULurl@#1{\hyper@linkurl{\ULurl@@{#1}}{#1}}
\DeclareRobustCommand*\ULurl{\hyper@normalise\ULurl@}
\makeatother

\def\lessim{\ \lower4pt\hbox{$
		\buildrel{\displaystyle <}\over\sim$}\ }
\def\gessim{\ \lower4pt\hbox{$\buildrel{\displaystyle >}
		\over\sim$}\ }

\usepackage{amsfonts,txfonts,psfrag}

\usepackage[mathscr]{euscript}

\usepackage{mdframed}

\numberwithin{equation}{section}

\makeatletter
\renewcommand{\@secnumfont}{\bfseries}
\makeatother

\newtheorem{thm}{Theorem}[section]    
         
\newtheorem{prop}[thm]{Proposition}

\theoremstyle{definition}

\newcommand{\Z}{\mathbb{Z}}

\newcommand{\R}{\mathbb{R}}

\newcommand{\B}{\textbf}

\newcommand{\eps}{\epsilon}

\newcommand{\1}{\B{\rm \B{1}}}

\newcommand{\E}{\mathbb{E}}

\newcommand{\prob}{\mathbb{P}}

\newtheorem{claim}{Claim}
\newtheorem{lemma}{Lemma}

\DeclareMathOperator*{\argmin}{arg\,min}

\DeclareMathOperator{\calP}{\mathcal{P}}

\begin{document}


\title{Equilibrium Distributions for t-distributed Stochastic Neighbour Embedding}


\author{Antonio Auffinger}
\address{Department of Mathematics, Northwestern University}
\email{tuca@northwestern.edu}


\author{Daniel Fletcher}
\address{Department of Mathematics, Northwestern University}
\email{fletcher@math.northwestern.edu}



\begin{abstract}
We study the empirical measure of the output of the t-distributed stochastic neighbour embedding algorithm when the initial data is given by $n$ independent, identically distributed inputs. We prove that under certain assumptions on the distribution of the inputs, this sequence of measures converges to an equilibrium distribution, which is described as a solution of a variational problem. 
\end{abstract}

\maketitle


\section{Introduction}

First introduced in 2008 by van der Maaten and Hinton \cite{JMLR:v9:vandermaaten08a} as a variation over the earlier stochastic neighbour embedding (SNE) method\cite{hinton2002stochastic}, t-distributed stochastic neighbour embedding (t-SNE) is a dimension reduction technique that is particularly well-suited for visualization of high-dimensional datasets. This technique has seen a wide and phenomenal use of applications in a range of scientific disciplines, including biology \cite{biology-example, biology-example2}, engineering \cite{engineering-example} and physics \cite{physics-example}. Despite the impressive empirical performance of t-SNE, and improvements in the implementation of the algorithm through various mechanisms \cite{Linderman2019, KMJ, JMLR:v15:vandermaaten14a}  so far there has been relatively little progress in mathematics in trying to rigorously understand the theory behind t-SNE.

Questions such as ``What happens when the initial dataset is pure noise?'', or ``How strong should a signal be  to visualize multiple clusters?'' or ``When  does the method work (and why does it work so well)?'' are surprisingly underdeveloped from a rigorous mathematical perspective. The goal of this paper is to address the first of these questions in the case where the number of data points diverges. Roughly speaking, in our main result, under some mild assumptions on the noise distribution, we prove the existence of an equilibrium measure (or a null-distribution) for t-SNE as the number of inputs goes to infinity.

Let $d, s \geq 2$ and $n\geq 1$.  The goal of the t-SNE algorithm is to find a collection of vectors $Y_1, \ldots, Y_n$ in $\mathbb R^s$ that is a low dimensional representation of input vectors $X_1, \ldots, X_n \in \mathbb R^d$. This is done by constructing associated measures as follows. Given $X_1, \ldots, X_n \in \mathbb R^d$, for $1\leq i, j \leq n$, we set 

\begin{equation}\label{eq:pij}
    p_{j|i} 
    = 
    \frac{
        \exp(-\lVert X_i - X_j \rVert^2 / 2 \sigma_i^2)
    }{
        \sum_{k \neq i} \exp(-\lVert X_i - X_k \rVert^2 / 2 \sigma_i^2)
    },
\end{equation}
where the sequence of numbers $\sigma_i$ satisfies for each $i = 1, \ldots, n,$
\begin{equation}\label{existencePerp}
    -\sum_{j =1, j\neq i}^n p_{j|i} \log p_{j|i} = \log \text{Perp}.
\end{equation}
The parameter $\text{Perp} = \text{Perp}(n)$ is called the perplexity of the model and we will assume that 
\begin{align}\label{eq:Perpl}
\text{Perp}(n) = n \rho
\end{align}
for some $\rho \in (0,1)$.
We then let for $1\leq i \neq j \leq n$,

\begin{align*}
    p_{ij}  = 
    \frac{p_{i|j} + p_{j|i}}{2n} 
    = 
    \frac{1}{2n} \left(
        \frac{
            \exp(-\lVert X_i - X_j \rVert^2 / 2 \sigma_i^2)
        }{
            \sum_{k \neq i} \exp(-\lVert X_i - X_k \rVert^2 / 2 \sigma_i^2)
        } 
        + 
        \frac{
            \exp(-\lVert X_i - X_j \rVert^2 / 2 \sigma_j^2)
        }{
            \sum_{k \neq j} \exp(-\lVert X_k - X_j \rVert^2 / 2 \sigma_j^2)
        } 
    \right),
\end{align*}
and we set $p_{ii}=0$.

For a collection of vectors $Y_1, \ldots, Y_n$ in $\mathbb R^s$ we  define for $1 \leq i \neq j \leq n,$
\begin{align*}
    q_{ij}=
    \frac{
        (1 + \lVert Y_i - Y_j \rVert^{2})^{-1}}{
        \sum_{k \neq \ell} (1 + \lVert Y_k - Y_\ell \rVert^{2})^{-1}}.
\end{align*}
and $q_{ii}=0$.  

Last, we define for $X=(X_1, \ldots, X_n) \in (\mathbb R^d)^n$ and $Y=(Y_1, \ldots, Y_n) \in \mathbb (\mathbb R^s)^n,$
\begin{align}\label{eq:energy}
    L_{n, \rho}(X, Y) 
    = 
    \sum_{1\leq i \neq j \leq n} p_{ij} \log\left( \frac{p_{ij}}{q_{ij}} \right).
\end{align}
The function $L_{n,\rho}$  is the relative entropy between the probability distributions constructed from the collections $(p_{ij})_{1\leq i,j \leq n}$ and $(q_{ij})_{1\leq i ,j \leq n}$. The output of t-SNE is the minimizer of \eqref{eq:energy}:
\begin{align}\label{eq:variation}
Y^* = \argmin_{Y \in (\mathbb R^s)^n} L_{n,\rho}(X, Y).
\end{align}
Existence of solutions to the problems in \eqref{existencePerp} and \eqref{eq:variation} are discussed in Section \ref{sec:properties_of_F_and_sigma}.

We denote $\mathcal M(\mathbb R^d)$ as the space of Borel probability measures on $\mathbb R^d$. Let $\mu_X \in  \mathcal M(\mathbb R^d)$ and $X_1, \ldots, X_n$ be independent, identically distributed random variables with common probability measure $\mu_X$. Let $Y^*= (Y_1, \dots, Y_n)$ be the output given by t-SNE and define the empirical measure 
\begin{equation}\label{empiricalmeasure}
    \mu_n = \frac{1}{n} \sum_i \delta_{(X_i, Y_i)} \in \mathcal M(\mathbb R^{d} \times \mathbb R^{s}).
\end{equation}
We now introduce the setup that describes the limit measures for the sequences $\mu_{n}$.  

Let $\mu \in \mathcal M (\mathbb R^{d} \times \mathbb R^{s})$. For $\sigma \in \R$, $\rho \in (0,1)$  define $F_{\rho, \mu}:\mathbb R^{d} \times \mathbb R \to \mathbb R$  as
\begin{align*}
    F_{\rho ,\mu}(x,\sigma) 
    & = 
    \int 
    \frac{
        \exp(-\lVert x - x' \rVert^2 / 2 \sigma^2)
    }{
        \int \exp(-\lVert x - x' \rVert^2 / 2 \sigma^2) \, d\mu(x')
    } 
    \log\left( 
        \frac{
            \exp(-\lVert x - x' \rVert^2 / 2 \sigma^2)
        }{
            \int \exp(-\lVert x - x' \rVert^2 / 2 \sigma^2) \, d\mu(x')
        } 
    \right) \, d\mu(x') 
    + \log \rho.
\end{align*}
Note that the integrations above are on the space $\mathbb R^{d} \times \mathbb R^{s}$. The integrands are constant in the remaining last  $s$ variables and the notation $d\mu(x')$ is to highlight that we are integrating on the variable $x' \in \mathbb R^{d}$.

Next, let $\sigma^*_{\rho , \mu}: \R^d \to \R$ be defined as the unique solution of the equation
\begin{equation}\label{SolutionSigma}
    F_{\rho , \mu}(x, \sigma^*_{\rho , \mu}(x)) = 0.
\end{equation}
Existence, uniqueness and properties of $\sigma_{\rho, \mu}^*$ are established in Section \ref{sec:properties_of_F_and_sigma}.

Now, given a function $\psi : \R^d \to \R^{+}$, define the function $p_{\psi}:\mathbb R^{d} \times \mathbb R^{d} \to \mathbb R$ as 
\begin{align*}
    p_{\psi}(x,x')
    & = 
    \frac{1}{2}
    \left(
        \frac{
            \exp(-\lVert x - x' \rVert^2 / 2 \psi(x)^2)
        }{
            \int \exp(-\lVert x - x' \rVert^2 / 2 \psi(x)^2) \, d\mu(x')
        }
        + 
        \frac{
            \exp(-\lVert x - x' \rVert^2 / 2 \psi(x')^2)
        }{
            \int \exp(-\lVert x - x' \rVert^2 / 2 \psi(x')^2) \, d\mu(x)
        }
    \right)
  \end{align*}
and also let $q:\mathbb R^{s}\times \mathbb R^{s} \to \mathbb R$ be given by
\begin{equation}
    q(y,y') = 
    \frac{
        (1 + \lVert y - y' \rVert^2)^{-1}
    }{
        \iint (1 + \lVert y - y' \rVert^2)^{-1} \, d\mu(y) d\mu(y')}.
\end{equation}

Lastly, define 
\begin{equation}\label{eq:Ifunctional}
    I_{\rho}(\mu)
    =
    \iint 
    p_{\sigma^*_{\rho , \mu}}(x,x')
    \log \left(
        \frac{p_{\sigma^*_{\rho , \mu}}(x,x')}{q(y,y')}
    \right)
    \, d\mu(x,y) d\mu(x',y')
\end{equation}
and for a given $\mu_{X} \in \mathcal M (\mathbb R^{d})$ write $\calP_{X}$ for the set of probability measures on $\R^d \times \R^s$ whose marginal coincides with $\mu_{X}$, that is,
\[
    \calP_X
    =
    \{
        \mu \in \mathcal M(\mathbb R^{d} \times \mathbb R^{s}): 
        \text{the } \mathbb R^{d} \text{ marginal of } \mu \text{ is } \mu_X
    \}.
\]
Moreover, define
\[
    \widetilde{\calP}_X
    =
    \left\{
        \mu \in \calP_X: 
        \int 
            (p(x,x') - q_{}(y,y'))
            \frac{y - y'}{1 + \lVert y - y' \rVert^2}
        \, d\mu(x',y') 
        = 0, 
        \, \mu(x,y)\text{-a.e.}
    \right\}.
\]
We can now state our main result.

\begin{thm}\label{thm:precise_theorem}
Assume $\mu_{X}$ has a $C^{1}$ density $f(x)$ and compact support. Let $X_i \in \R^d$ be independent, identically distributed random variables with common distribution $\mu_{X}$.  
Then for any $\rho \in (0,1)$,
\[
    \lim_{n \to \infty} \inf_Y L_{n, \rho}(X, Y)
    =
    \inf_{\mu \in \widetilde\calP_X} I_{\rho}(\mu).
\]
Moreover, there exist a measure $\mu^*$ and a sub-sequence $n_k$ such that $\mu_{n_k}$ converges weakly to $\mu^*$ and $I_{\rho}(\mu^*) = \inf_{\mu \in \widetilde\calP_X} I_{\rho}(\mu)$.
\end{thm}

We can also infer that the limiting measure has compact support. 

\begin{thm}\label{thm2} Under the same assumptions of Theorem \ref{thm:precise_theorem}, $\mu^*$ has compact support.
\end{thm}

\begin{figure}\label{SimFig}
\includegraphics[width=0.5\textwidth]{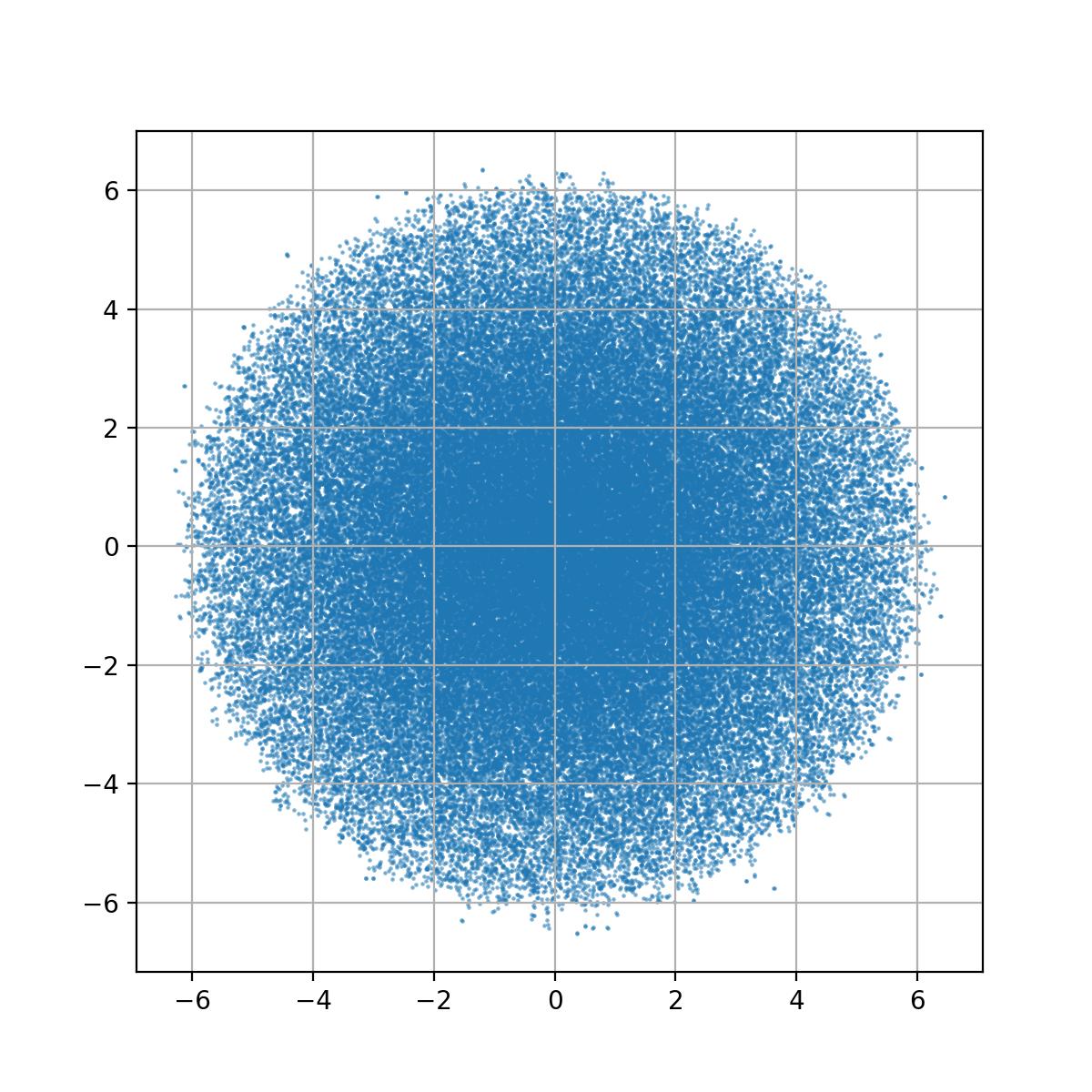}
\caption{Simulation of a t-SNE output with $100,000$ independent input data distributed according to an uniform distribution on $[-1,1]^{d}$ in dimension $500$.}
\end{figure}

We conclude this section with some historical remarks and a brief description on the structure of the paper. As far as we know, Theorem \ref{thm:precise_theorem} is the first result in the literature that characterizes the null distribution of t-SNE. Given the exceptional empirical performance of t-SNE, it is surprising to us that theoretical results are still scarce. Significant achievements were made by Linderman and Steinerberger \cite{linderman-steinerberger} after earlier work by Shaham and Steinerberger \cite{SSArxiv}, who gave conditions under which the algorithm successfully separates clusters (in a suitably formal sense). More exploration on this question was performed by Arora, Hu and Kothari  \cite{pmlr-v75-arora18a} who studied performance of the gradient descent performed by t-SNE under certain deterministic conditions on the ground-truth clusters. More recently, Cai and Ma \cite{cai-ma-2022} formalized and made rigorous an observation first made by Linderman and Steinerberger \cite{linderman-steinerberger}: a relationship to spectral clustering when t-SNE is ran in what is called the early exaggeration phase. For a comprehensive survey on other aspects of stochastic neighbor embeddings and their relationship to other dimension reduction techniques readers are invited to check Chapter 15 in \cite{book2} and the references therein.

The paper is organized as follows. In the next section, we show existence and uniqueness of the solution of \eqref{SolutionSigma}. This section is where we use the assumptions of a $C^{1}$ density and compact support on $\mu_{X}$. In Section \ref{sec3}, we show uniform convergence of the sum \eqref{eq:energy} towards $F_{\rho,\mu}$ and in Section 4, we prove Theorem \ref{thm:precise_theorem}. Theorem \ref{thm2} is proven in Section 5. The Appendix contains a separate analysis of the space $ \widetilde{\calP}_X$ which may be of independent interest.

%
%

\subsection{Acknowledgments} Antonio Auffinger's research is partially supported by NSF Grant CAREER DMS-1653552, Simons Foundation/SFARI (597491-RWC), and NSF Grant 2154076. Daniel Fletcher's  research is partially supported by the Simons Foundation/SFARI (597491-RWC). Both authors would like to thank Eric Johnson, Madhav Mani,  and William Kath for inspiring discussions on dimension reduction methods.

\section{Existence and properties of $F_{\rho , \mu}$, $\sigma^*_{\rho , \mu}$}\label{sec:properties_of_F_and_sigma}

In this section, we establish existence and uniqueness of solutions of the problem \eqref{SolutionSigma}.

Note that $F_{\rho,\mu}$ is $C^{\infty}$ and we can write
\begin{align*}
    F_{\rho , \mu}(x,\sigma) = &-\int
    \frac{\lVert x - x' \rVert^2}{2 \sigma^2}
    \frac{
        \exp(-\lVert x - x' \rVert^2 / 2 \sigma^2)
        }{
        \int \exp(-\lVert x - x' \rVert^2 / 2 \sigma^2) \, d\mu(x')
        } 
    \, d\mu(x') \\
    &- 
    \log\left(
        \int \exp(-\lVert x - x' \rVert^2 / 2 \sigma^2) \, d\mu(x')
    \right)
    + \log \rho.
\end{align*}

\subsection{Uniqueness and smoothness of $\sigma^*_{\rho , \mu}$}\label{subsec:uniqueness-of-sigma}

We start with the following Lemma that establishes lower bounds for $F_{\rho,\mu}$. 
\begin{lemma}\label{lemma_F_is_unbounded}
If $\mu$ has a $C^{1}$ density $f$ such that both $f$ and $\lVert \nabla f \rVert$ are bounded then there exists a constant $C>0$ so that 
\[
    F_{\rho, \mu}(x,\sigma) \geq -C \log \sigma - \log f(x) + \log \rho.
\]
In particular, $F_{\rho, \mu}(x,\sigma) \to \infty$ as $\sigma \to 0$.
\end{lemma}

\begin{proof}
First note that
\begin{align*}
    \int \exp(-\lVert x - x' \rVert^2 / 2 \sigma^2) \, d\mu(x')
    & = 
    \left(2 \pi \sigma^2 \right)^{d/2}
    \int
        \frac{e^{-\lVert x - x' \rVert^2 / 2 \sigma^2} }{\left( 2 \pi \sigma^2 \right)^{d/2}}
        f(x') \, 
    dx'
    = 
    \left( 2 \pi \sigma^2 \right)^{d/2}
    \E_{x,\sigma} f(X')    
\end{align*}
where the expectation is over a $N(x,\sigma I)$ random variable.
As $\sigma \to 0$, these Gaussian measures converge weakly to $\delta_x$.
Hence since $f$ is continuous and bounded $\E f(X') \to f(x)$.
In particular, for some $\varepsilon(\sigma)$ that tends to 0 as $\sigma \to 0$, we have
\begin{align}
    -\log \int \exp(\lVert x - x' \rVert^2/2\sigma^{2}) \, d\mu(x')
    & =
    -\frac{d}{2} \log (2 \pi \sigma^{2}) - \log f(x) + \varepsilon(\sigma) \nonumber
    \\
    & \geq 
    -C \log \sigma - \log f(x) \label{eq:lem1ineq}.
\end{align}

Now using integration by parts, for all $x$ in the support of $\mu$ (so that $f(x) > 0$), we have:
\begin{align*}
    & \quad
    \frac{
        \int \frac{\lVert x - x' \rVert^2}{2\sigma^2} \exp(-\lVert x - x' \rVert^2 / 2 \sigma^2) \, d\mu(x')
    }{
        \int \exp(-\lVert x - x' \rVert^2 / 2\sigma^2) \, d\mu(x')
    }
    \\
    & = 
    - \frac{ \frac{1}{4}
        \int
        \nabla \lVert x - x' \rVert^2 
        \cdot 
        \nabla \exp(-\lVert x - x' \rVert^2 / 2 \sigma^2) 
        f(x') \, dx'
    }{
        \int \exp(-\lVert x - x' \rVert^2 / 2 \sigma^2) f(x') \, dx'
    }
    \\
    & = 
    \frac{1}{4}
    \frac{
        \int
        \exp(-\lVert x - x' \rVert^2 / 2 \sigma^2) 
        \left(
            \nabla f(x') \cdot \nabla \lVert x - x' \rVert^2
            +
            f(x') \Delta \lVert x - x' \rVert^2
        \right)
        \, dx'
    }{
        \int \exp(-\lVert x - x' \rVert^2 / 2 \sigma^2) f(x') \, dx'
    }
    \\
    & =
    \frac{1}{2}
    \frac{
        \int
        \exp(-\lVert x - x' \rVert^2 / 2 \sigma^2) 
        \nabla f(x') \cdot (x - x')
        \, dx'
    }{
        \int \exp(-\lVert x - x' \rVert^2 / 2 \sigma^2) f(x') \, dx'
    }
    +
    \frac{d}{2}
    \\
    & = 
    \frac{1}{2}
    \frac{
        \E_{x,\sigma} \left[ 
            \nabla f(X') \cdot (x - X')
        \right]
    }{
        \E_{x,\sigma} f(X')
    }
    + 
    \frac{d}{2}
    \\
    & \to 
    \frac{d}{2},
\end{align*}
which combined with \eqref{eq:lem1ineq} ends the proof.
\end{proof}

For convenience, define
\[
    \widetilde{F}_{\rho , \mu}(x,\eta) 
    = 
    -\int
    \frac{
        \eta \lVert x - x' \rVert^2 
        \exp(-\eta \lVert x - x' \rVert^2)
        }{
        \int \exp(-\eta \lVert x - x' \rVert^2) \, d\mu(x')
        } 
    \, d\mu(x') 
    - 
    \log\left(
        \int \exp(-\eta\lVert x - x' \rVert^2) \, d\mu(x')
    \right) 
    + \log \rho.
\]
Then we have
\[
    F_{\rho, \mu}(x, \sigma) 
    =
    \widetilde{F}_{\rho, \mu}(x, 1 / 2 \sigma^2).
\]

\begin{prop}\label{prop:uniqueness_of_sigma} Let $\mu \in \mathcal M (\mathbb R^{d})$ such that $\mu$ has a $C^{1}$ density $f$ such that both $f$ and $\lVert \nabla f \rVert$ are bounded.  The equation 
$F_{\rho ,\mu}(x,\sigma)=0$ has a unique solution $\sigma^*_{\rho ,\mu}(x)$ for all $\rho \in (0,1)$ and $x \in \mathbb R^{d}$.
Moreover, $x \mapsto \sigma^*_{\rho , \mu}(x)$ is smooth.
\end{prop}

\begin{proof}
Note that
\begin{align*}
    \partial_{\eta} \widetilde{F}_{\rho , \mu}(x, \eta)
    & =
    -\int \frac{\lVert x - x' \rVert^2 \exp(-\eta\lVert x - x' \rVert^2)}{\int \exp(-\eta\lVert x - x' \rVert^2) \, d\mu(x')} \, d\mu(x')
    +
    \int \frac{\eta \lVert x - x' \rVert^4 \exp(-\eta\lVert x - x' \rVert^2)}{\int \exp(-\eta\lVert x - x' \rVert^2) \, d\mu(x')} \, d\mu(x')
    \\
    & \quad
    - 
    \eta \frac{\left(\int \lVert x - x' \rVert^2 \exp(-\eta\lVert x - x' \rVert^2) \, d\mu(x') \right)^2}{\left( \int \exp(-\eta\lVert x - x' \rVert^2) \, d\mu(x') \right)^2}
    +
    \frac{\int \lVert x - x' \rVert^2 \exp(-\eta\lVert x - x' \rVert^2) \, d\mu(x')}{\int \exp(-\eta\lVert x - x' \rVert^2) \, d\mu(x')}
    \\
    & = 
    \eta \left( \int \frac{\lVert x - x' \rVert^4 \exp(-\eta\lVert x - x' \rVert^2)}{\int \exp(-\eta\lVert x - x' \rVert^2) \, d\mu(x')} \, d\mu(x') 
    -
    \frac{\left(\int \lVert x - x' \rVert^2 \exp(-\eta\lVert x - x' \rVert^2) \, d\mu(x') \right)^2}{\left( \int \exp(-\eta\lVert x - x' \rVert^2) \, d\mu(x') \right)^2} \right).
\end{align*}
Above, we can differentiate under the integral by Dominated Convergence Theorem.

Applying Cauchy-Schwarz with 
\begin{align*}
    f(x') 
    = 
    \frac{
        \lVert x - x' \rVert^2 \exp(-\eta \lVert x - x' \rVert^2 / 2)
    }{
        \left( \int \exp(-\eta\lVert x - x' \rVert^2) \, d\mu(x') \right)^{1/2}
    }, 
    \quad g(x') 
    = 
    \frac{
        \exp(-\eta \lVert x - x' \rVert^2 / 2)
    }{
        \left( \int \exp(-\eta\lVert x - x' \rVert^2) \, d\mu(x') \right)^{1/2}
    }
\end{align*}
we see that $\partial_{\eta} \widetilde{F}_{\rho,\mu} \geq 0$ with equality if and only if $\eta = 0$.

Now,
\[
    \partial_{\sigma} F_{\rho, \mu}(x, \sigma) 
    = 
    -\frac{1}{\sigma^3} \partial_{\eta} \widetilde{F}_{\rho, \mu}(x, 1/2\sigma^2).
\]
So we see that $\partial_{\sigma} F_{\rho, \mu} \leq 0$ for all $\sigma \neq 0$.
Moreover,  we have 
\[
    \lim_{\sigma \to \infty} F_{\rho , \mu}(x,\sigma) 
    = 
    \log \rho < 0.
\]

By Lemma \ref{lemma_F_is_unbounded} above
\[
    \lim_{\sigma \to 0} F_{\rho, \mu}(x, \sigma) 
    = 
    \infty
\]
so there exists unique $\sigma(x) = \sigma^*_{\rho ,\mu}(x)$ such that $F_{\rho ,\mu}(x,\sigma^*_{\rho , \mu}(x)) = 0$.
Moreover, by the implicit function theorem, $\sigma^*_{\rho , \mu}: \R^d \to \R$ is smooth. 
\end{proof}

\section{Uniform convergence of $F_{\rho, \mu_n}(x, \sigma)$}\label{sec3}

Recall the empirical measures $\mu_{n}$ defined in \eqref{empiricalmeasure}. In this section, we establish uniform convergence of $F_{\rho,\mu_n}(x, \sigma)$. This result is useful in its own right and holds under the lighter assumption that $\mu$ has sub-Gaussian tails. To lighten the notation, we will remove the subscript $\rho$ from $F_{\rho, \mu_n}(x, \sigma)$. We start this section with some lemmas on measures with sub-Gaussian tails.

Recall that we say that a measure $\nu$ on $\mathbb R$ has sub-Gaussian tails if there is a positive constant $C$ such that $\nu((t,\infty)) \leq 2\exp (-Ct^{2})$. 

\begin{lemma}\label{lemma:sub_gaussian_norm_tail_bound}
Let $\nu$ be sub-Gaussian with second moment equal to 1 and suppose $X \in \R^d$ has i.i.d. entries drawn from $\nu$.
Then there exists $C, c_1$ such that
\[
    \prob\left(\lVert X \rVert \geq \sqrt{d} + t\right)
    \leq    
    Ce^{-c_1 t^2}.
\]  
\end{lemma}

\begin{proof}
Since $X_1, \dots, X_d$ are sub-Gaussian, $X_i^2 - 1$ is sub-exponential for each $1\leq i \leq d$.
It follows from Bernstein's inequality \cite[Corollary 2.11]{BLMbook} that we have
\[
    \prob\left( 
        \left\vert 
            \frac{1}{d} \lVert X \rVert^2 - 1
        \right\vert
        \geq 
        u
    \right)
    \leq 
    C \exp\left(-c_1 d \min(u, u^2)\right).
\]
Now, note that for all $z \geq 0$, $\lvert z - 1 \rvert \geq a$ implies that $\lvert z^2 - 1 \rvert \geq \max(a, a^2)$ (separate into cases $a \leq 1, a \geq 1$).
In particular, by the above tail inequality with $u = \max(a,a^2)$, we have
\[
    \prob\left(
        \left\vert
            \frac{1}{\sqrt{d}} \lVert X \rVert - 1
        \right\vert
        \geq 
        a
    \right)
    \leq 
    C \exp\left(-c_1 d a^2\right).
\]
Setting $t = a\sqrt{d}$ we see that
\[
    \prob\left(
        \lVert X \rVert \geq \sqrt{d} + t
    \right)
    \leq
    C e^{-c_1 t^2}.
\]
\end{proof}

\begin{lemma}\label{lemma:max_sub_gaussian_norm_tail_bound}
Let $\nu$ be sub-Gaussian and $c_{1}$ be the constant from Lemma \ref{lemma:sub_gaussian_norm_tail_bound}. 
Let $X_1, \dots, X_n \in \R^d$ have i.i.d. entries drawn from $\nu$.
Then there exists a constant $C$ such that
\[
    \prob\left(
        \max_{i=1,\dots,n} 
        \lVert X_i \rVert 
        \geq 
        \sqrt{\frac{\log n}{c_1}}
        +
        \sqrt{d}
        + 
        t
    \right)
    \leq 
    C e^{-c_1 t^2}.
\]
\end{lemma}

\begin{proof}
By Lemma \ref{lemma:sub_gaussian_norm_tail_bound}, we have
\begin{align*}
    \prob\left(
        \max_{i=1,\dots,n} \lVert X_i \rVert
        \geq 
        \sqrt{d} + u
    \right)
    & \leq 
    \sum_i
    \prob\left(
        \lVert X_i \rVert 
        \geq
        \sqrt{d} + u
    \right)
    \\
    & = 
    n \prob\left(
        \lVert X_1 \rVert 
        \geq
        \sqrt{d} + u
    \right)
    \\
    & \leq 
    Cn e^{-c_1 u^2}.
\end{align*}
Now, setting $u = \sqrt{\frac{\log n}{c_1}} + t$,
\[
    e^{-c_1 u^2}
    \leq
    \frac{1}{n} e^{-c_1 t^2}.
\]
Hence
\[
    \prob\left(
        \max_{i=1,\dots,n} 
        \lVert X_i \rVert 
        \geq 
        \sqrt{\frac{\log n}{c_1}}
        +
        \sqrt{d}
        + 
        t
    \right)
    \leq 
    C e^{-c_1 t^2}.
\]
\end{proof}

We will now state the main result of this section. For its proof we still need a few more Lemmas. 

\begin{prop}\label{prop_convergence_of_F}
Let $\nu$ be sub-Gaussian and $c_{1}$ be a constant from Lemma \ref{lemma:sub_gaussian_norm_tail_bound}. 
Let $X_1, \dots, X_n \in \R^d$ have i.i.d. entries drawn from $\nu$ and let $\mu_n = \frac{1}{n} \sum \delta_{X_i}$.
Let $\delta > 0$ be a constant such that $c_1 \delta^2 > 16$.
Then, there exists $\epsilon > 0$ such that almost surely in $X_1, \dots, X_n$:
\[
    \sup_{\substack{
            \sigma \in [\delta, \sqrt{n}] \\
            i = 1, \dots, n
        }
    }
    \left\vert 
        F_{\mu_n}(X_i, \sigma)
        -
        F_{\mu}(X_i, \sigma)
    \right\vert
    =
    O(n^{-\epsilon}).
\]
\end{prop}

\begin{lemma}\label{lemma:rademacher_bound}
Suppose that $\varepsilon_1, \dots, \varepsilon_n$ are i.i.d. Rademacher and let $B \subseteq \R^n$ be a finite set.
Then
\[
    \E \sup_{b \in B} \sum_{i=1}^n \varepsilon_i b_i 
    \leq 
    \max_{b \in B} \lVert b \rVert_2 \sqrt{2 \log \lvert B \rvert}.
\]
\end{lemma}

\begin{proof}[Proof of Lemma \ref{lemma:rademacher_bound}]
Let $\lambda \in \R$.
Then
\begin{align*}
    \exp\left( 
        \lambda \E \sup_b \sum_{i=1}^n \varepsilon_i b_i
    \right)
    \leq 
    \E \exp\left( 
        \lambda \sup_b \sum_{i=1}^n \varepsilon_i b_i 
    \right)
   = 
    \E \sup_b \exp\left( 
        \lambda \sum_{i=1}^n \varepsilon_i b_i
    \right)
 \leq 
    \sum_b \E \exp\left( 
        \lambda \sum_{i=1}^n \varepsilon_i b_i
    \right).
\end{align*}
Now note that
\begin{align*}
    \E \exp\left( 
        \lambda \sum_{i=1}^n \varepsilon_i b_i 
    \right)
     = 
    \prod_i \E \exp\left( 
        \lambda \varepsilon_i b_i
    \right)
    = 
    \prod_i \cosh(\lambda b_i)
    \leq 
    \prod_i \exp\left( \frac{\lambda^2 b_i^2}{2} \right)
    = 
    \exp\left( \frac{\lambda^2 \lVert b \rVert^2}{2} \right).
\end{align*}
In particular, we see that
\begin{align*}
    \exp\left( 
        \lambda \E \sup_b \sum_{i=1}^n \varepsilon_i b_i
    \right)
     \leq 
    \lvert B \rvert \exp\left(
        \frac{\lambda^2 \max_b \lVert b \rVert^2}{2}
    \right)
 \end{align*}
and therefore
\begin{align*} 
    \E \sup_b \sum_{i=1}^n \varepsilon_i b_i
     \leq 
    \frac{\log \lvert B \rvert}{\lambda}
    + 
    \frac{\lambda \max_b \lVert b \rVert^2}{2}.
\end{align*}
This upper bound is minimized at 
\[
    \lambda 
    = 
    \frac{
        \sqrt{2 \log \lvert B \rvert}
    }{
        \max_b \lVert b \rVert_2
    }
\]
which gives
\[
    \E \sup_b \sum_{i=1}^n \varepsilon_i b_i
    \leq 
    \max_{b \in B} \lVert b \rVert_2 \sqrt{2 \log \lvert B \rvert}.
\]
\end{proof}




\begin{lemma}\label{lemma:general-uniform-convergence}
Let $\delta > 0$ be constant, $K_n = [\delta, \sqrt{n}]$ and $h: \mathbb R^{d}\times \mathbb R^{d} \times \mathbb R \to \mathbb R$.
Suppose the following hold:
\begin{itemize}
    \item $\mu$ is sub-Gaussian,
    \item $h(x,x',\sigma)$ is $C^1$ for all $x, x'$ and all $\sigma \in K_n$,
    \item $h(x,x',\sigma)$ and $\nabla_x h(x,x',\sigma)$ are bounded for all $x, x'$ and all $\sigma \in K_n$,
    \item $h(x,x',\sigma) \to 0$ as $\lVert x - x' \rVert / \sigma \to \infty$.
\end{itemize}
Let $X_1, \dots, X_n$ be i.i.d. random variables drawn from $\mu$.
Let $\mu_n = \frac{1}{n} \sum_i \delta_{X_i}$.
Then for any $\varepsilon >0$ with probability at least $1 - e^{-c \varepsilon^2 n}$ we have
\[
    \sup_{\sigma \in K_n, x}
        \left\vert
            \int h(x, x', \sigma) \, d\mu_n(x') 
            - 
            \int h(x, x', \sigma) \, d\mu(x')
        \right\vert 
    \leq
    \varepsilon 
    +
    O\left(
       \sqrt{\frac{\log n}{n}}
    \right).
\]
\end{lemma}

\begin{proof}
Let $\E$ be the expectation taken over the randomness in $X_1, \dots, X_n$.
We can write
\begin{align*}
    \left\vert
        \int h(x,x',\sigma) \, d\mu_n(x') - \int h(x,x',\sigma) \, d\mu(x')
    \right\vert
    &=
    \left\vert
        \frac{1}{n} \sum_i h(x,X_i,\sigma) - \int h(x,x',\sigma) \, d\mu(x')
    \right\vert
    \\
    & =
    \left\vert
        \frac{1}{n} \sum_i h(x,X_i,\sigma) - \E h(x,X_i,\sigma)
    \right\vert.
\end{align*}
Now, since $h$ is bounded, $\frac{1}{n} \sum_i h(x,X_i,\sigma) - \E h(x,X_i,\sigma)$ has bounded differences (as a function of $X_1, \dots, X_n$) with $c_i = C/n$.
Hence, by McDiarmid's inequality \cite[Theorem 6.2]{BLMbook}, with probability at least $1 - e^{-c\varepsilon^2 n}$, we have
\begin{align}\label{coffeemug}
    \left\vert
        \frac{1}{n} \sum_i h(x,X_i,\sigma) - \E h(x,X_i,\sigma)
    \right\vert
    & \leq 
    \E \left\vert
        \frac{1}{n} \sum_i h(x,X_i,\sigma) - \E h(x,X_i,\sigma)
    \right\vert
    +
    \varepsilon \nonumber
    \\
    & \leq
    \E \sup_{\sigma \in K_n, x}
        \left\vert
            \frac{1}{n} \sum_i h(x,X_i,\sigma) - \E h(x,X_i,\sigma)
        \right\vert
    +
    \varepsilon.
\end{align}
Now, let $X'_i$ be i.i.d. drawn from $\mu$ and independent from $X_i$ and let $\E'$ be expectation over all of $X_i'$.
Then
\begin{align*}
    \E \sup_{\sigma \in K_n, x}
        \left\vert
            \frac{1}{n} \sum_i h(x,X_i, \sigma) - \E h(x,X_i, \sigma)
        \right\vert 
    & =
    \E \sup_{\sigma \in K_n, x}
        \left\vert
            \frac{1}{n} \sum_i h(x,X_i, \sigma) - \E' h(x,X'_i,\sigma)
        \right\vert
    \\
    & = 
    \E \sup_{\sigma \in K_n, x}
        \left\vert
            \E' \frac{1}{n} \sum_i h(x,X_i,\sigma) - h(x,X'_i,\sigma)
        \right\vert
    \\ 
    & \leq 
    \E \E' \sup_{\sigma \in K_n, x}
        \left\vert
            \frac{1}{n} \sum_i h(x,X_i,\sigma) - h(x,X'_i,\sigma)
        \right\vert.
\end{align*}
Let $\varepsilon_i$ be i.i.d. Rademacher distributed random variables. 
Then since $X_i, X_i'$ are i.i.d., $h(x,X_i,\sigma) - h(x,X'_i,\sigma)$ and $\varepsilon_i (h(x,X_i,\sigma) - h(x,X'_i,\sigma))$ have the same distribution.
In particular, by the previous inequality:
\begin{align*}
    \E \sup_{\sigma \in K_n, x}
        \left\vert
            \frac{1}{n} \sum_i h(x,X_i,\sigma) - \E h(x,X_i,\sigma)
        \right\vert 
    & \leq 
    \E \E' \sup_{\sigma \in K_n, x}
        \left\vert
            \frac{1}{n} \sum_i \varepsilon_i (h(x,X_i,\sigma) - h(x,X'_i,\sigma))
        \right\vert
    \\
    & \leq 
    2 \E \sup_{\sigma \in K_n, x}
        \left\vert
            \frac{1}{n} \sum_i \varepsilon_i h(x,X_i,\sigma)
        \right\vert
    \\
    & = 
    2 \E \E \left[
         \sup_{\sigma \in K_n, x}
        \left\vert
            \frac{1}{n} \sum_i \varepsilon_i h(x,X_i,\sigma)
        \right\vert
        \big\vert
        X_1, \dots, X_n
    \right].
\end{align*}

Let
\[
    H(x,\sigma) 
    = 
    \left\vert
        \frac{1}{n} \sum_i \varepsilon_i h(x,X_i,\sigma).
    \right\vert
\]
and $E = \left\{\max_{i=1,\dots,n} \lVert X_i \rVert \geq 2 \sqrt{n} + \sqrt{\frac{\log n}{c_1}} \right\}$.
Since $\mu$ is sub-Gaussian, by Lemma \ref{lemma:max_sub_gaussian_norm_tail_bound} with $t=\sqrt{n}$, there exists $C$ such that $\prob(E) \leq Ce^{-c_1n}$.

Let $A^{(n)} = [-n,n]^d$.
Then on $E^c$, for $x \notin A^{(n)}$ and $\sigma \in K_n$, we have
\(
    \frac{\lVert x - X_i \rVert}{\sigma} 
    \geq
    \sqrt{n}/2
\)
and hence, by the decay condition on $h$, for sufficiently large $n$ we have
\[
    \sup_{x \in \R^d, \sigma \in K_n} H(x,\sigma)
    =
    \sup_{ x \in A^{(n)}, \sigma \in K_n} H(x,\sigma).
\]

Since $A^{(n)} \times K_n$ is compact and $H(x,\sigma)$ is continuous, $ \sup_{ x \in A^{(n)}, \sigma \in K_n} H(x,\sigma)$ is achieved at some $(x^*, \sigma^*)$.
Let $A^{(n)}_k = A^{(n)} \cap \frac{1}{k} \Z^d$ and let $x_k \in A^{(n)}_k$ be the closest point in $A^{(n)}_k$ to $x^*$.
Then, using the fact that $\nabla_x h(x,x',\sigma)$ is bounded:
\begin{align*}
    \left\vert 
        \sup_{\sigma \in K_n, x \in A^{(n)}} H(x,\sigma) 
        - 
        \sup_{x \in A_k^{(n)}} H(x,\sigma^*)
    \right\vert
    & = 
    H(x^*,\sigma^*) 
    - 
    \sup_{x \in A_k^{(n)}} H(x,\sigma^*)
    \\
    & \leq 
    H(x^*,\sigma^*) - H(x_k,\sigma^*)
    \\
    & \leq 
    \sup_{x \in A^{(n)}} \lVert \nabla_x H(x,\sigma^*) \rVert \, \lVert x^* - x_k \rVert
    \\
    & =
    O\left( \frac{1}{\sqrt{k}} \right).
\end{align*}

Hence combining the above inequalities, and using the fact that $h$ is bounded, we have
\begin{align*}
    &\E \left[ 
    \sup_{x \in \R^d, \sigma \in K_n}
        \left\vert
            \frac{1}{n} \sum_i h(x,X_i,\sigma) - \E h(x,X_i,\sigma)
        \right\vert
    \right]
    \\
    \leq \,
    & 2 \E \E \left[
        \1_{E^c} \sup_{x \in A^{(n)}_k}
        \left\vert
            \frac{1}{n} \sum_i \varepsilon_i h(x,X_i,\sigma^*)
        \right\vert
        \big\vert
        X_1, \dots, X_n
    \right]
    + 
    O\left(\frac{1}{\sqrt{k}}\right)
    +
    C \prob(E).
\end{align*}

Now apply Lemma \ref{lemma:rademacher_bound} with 
\[
    B 
    = 
    \left\{
        \left(\frac{1}{n} h(a, X_i, \sigma^*)\right)_{i=1,\dots,n}
        : 
        a \in A^{(n)}_k 
    \right\}.
\]
Then $\lvert B\rvert \leq \lvert A^{(n)}_k \rvert = (2nk)^d$ and since $h$ is bounded, $\max_{b \in B} \lVert b \rVert_2 = O\left(\frac{1}{\sqrt{n}}\right)$.
It follows that
\[
    \E \left[
        \1_{E^c} \sup_{x \in A^{(n)}_k}
        \left\vert
            \frac{1}{n} \sum_i \varepsilon_i h(x,X_i,\sigma^*)
        \right\vert
        \big\vert
        X_1, \dots, X_n
    \right]
    \leq 
    C \sqrt{\frac{d \log (nk)}{n}}.
\]
Taking $k = n$ we have 
\[
    \E \left[
    \sup_{x \in \R^d, \sigma \in K_n}
        \left\vert
            \frac{1}{n} \sum_i h(x,X_i,\sigma) - \E h(x,X_i,\sigma)
        \right\vert
    \right]
    = 
    O\left(\sqrt{\frac{d \log n}{n}} \right)
    + 
    O\left(\frac{1}{\sqrt{n}}\right)
    +
    Ce^{-cn}
    =
    O\left(\sqrt{\frac{\log n}{n}}\right).
\]

In particular, combining this bound with \eqref{coffeemug}, with probability at least $1 - e^{-c\varepsilon^2 n}$, 
\[
    \sup_{x \in \R^d,\sigma\in K_{n}}
        \left\vert
            \frac{1}{n} \sum_i h(x,X_i, \sigma) - \E h(x,X_i, \sigma)
        \right\vert
    \leq 
    \varepsilon + O\left(\sqrt{\frac{\log n}{n}}\right).
\]

\end{proof}


Define
\begin{align} \label{fandgsigma}
    f_{\sigma}(x,x') 
    = 
    \exp\left(
        -\lVert x - x' \rVert^2 / 2\sigma^2
    \right), 
    \quad
    g_{\sigma}(x,x')
    = 
    \frac{
        \lVert x - x' \rVert^2
    }{
        2\sigma^2
    }
    \exp\left(
        -\lVert x - x' \rVert^2 / 2\sigma^2
    \right).
\end{align}
\begin{lemma}\label{lemma:convergence-of-g/f}
Let $c_1$ be the decay constant from Lemma \ref{lemma:sub_gaussian_norm_tail_bound} and let $\delta > 0$.
Then, almost surely in $X_1, \dots, X_n$, 
\[
    \sup_{\substack{
            \sigma \in [\delta, \sqrt{n}] \\
            i = 1,\dots,n}
        }
    \left\vert
        \frac{
            \int g_{\sigma}(X_i,x') \, d\mu_n(x')
        }{
            \int f_{\sigma}(X_i,x') \, d\mu_n(x')
        } 
        -
        \frac{
            \int g_{\sigma}(X_i,x') \, d\mu(x')
        }{
            \int f_{\sigma}(X_i,x') \, d\mu(x')
        } 
    \right\vert  
    =
    O(n^{8/c_1\delta^2 - 1/2} \sqrt{\log n}).
\]
\end{lemma}

\begin{proof}
First note that
\begin{align*}
    \left\vert
        \frac{
            \int g_{\sigma}(x,x') \, d\mu_n(x')
        }{
            \int f_{\sigma}(x,x') \, d\mu_n(x')
        } 
        \right.&\left.-
        \frac{
            \int g_{\sigma}(x,x') \, d\mu(x')
        }{
            \int f_{\sigma}(x,x') \, d\mu(x')
        } 
    \right\vert 
   \leq 
    \left\vert
        \frac{
            \int g_{\sigma}(x,x') \, d\mu_n(x')
            -
            \int g_{\sigma}(x,x') \, d\mu(x')
        }{
            \int f_{\sigma}(x,x') \, d\mu_n(x')
        }
    \right\vert
    \\
    &+
    \left\vert 
        \frac{
            \int g_{\sigma}(x,x') \, d\mu(x')
            \left(
                \int f_{\sigma}(x,x') \, d\mu_n(x')
                -
                \int f_{\sigma}(x,x') \, d\mu(x')
            \right)
        }{
            \int f_{\sigma}(x,x') \, d\mu(x')
            \int f_{\sigma}(x,x') \, d\mu_n(x')
        }
    \right\vert.
\end{align*}

First we want to bound the terms in the denominators.
The terms in the numerators are dealt with using Lemma \ref{lemma:general-uniform-convergence}.
Let $\alpha >0$. Define
\[
    E_1' = 
    \left\{
        \max_{i=1,\dots,n} 
        \lVert X_i \rVert 
        \geq 
        \sqrt{\left(
                \alpha + \frac{1}{c_1}
            \right)}
        \sqrt{\log n}
        +
        \sqrt{d}
    \right\}.
\]
Setting $t = \sqrt{\alpha \log n}$ in Lemma \ref{lemma:max_sub_gaussian_norm_tail_bound}, we have $\prob(E_1') \leq Cn^{-c_1\alpha}$.
To absorb the $\sqrt{d}$ term in $E_1'$, we can find $\epsilon_n \to 0$ such that
\[
    E_1 := 
    \left\{
        \max_{i=1,\dots,n} 
        \lVert X_i \rVert 
        \geq 
        \sqrt{
                \alpha + \frac{1}{c_1} + \epsilon_n
        }
        \sqrt{\log n}
    \right\}
\]
has $\prob(E_1) \leq Cn^{-c_1\alpha}$.

On $E_1^c$ for all $i, j$, we have
\[
    \lVert X_i - X_j \rVert^2
    \leq 
    4\left(
        \alpha + \frac{1}{c_1} + \epsilon_n
    \right)
    \log n
    =: 4\gamma \log n.
\]

Then, for $\sigma \in [\delta, \sqrt{n}]$, on $E_1^c$ for all $i$ we have
\begin{align*}
    \int f_{\sigma}(X_i,x') \, d\mu_n(x')
    =
    \int
        \exp\left(
            -\lVert X_i - x' \rVert^2 / 2\sigma^2
        \right)
    \, d\mu_n(x')
    \geq 
    \exp\left(
        -\frac{2 \gamma \log n}{\delta^2}
    \right)
    =
    n^{-2\gamma / \delta^2}.
\end{align*}

Similarly, applying Lemma \ref{lemma:sub_gaussian_norm_tail_bound} with $t = \sqrt{\alpha \log n}$, for all $i$ we have that
\begin{align*}
    \int f_{\sigma}(X_i,x') \, d\mu(x')
    & \geq 
    \left(
        1 - Cn^{-c_1\alpha}
    \right)
    \exp\left(
        -\frac{2 \gamma \log n}{\delta^2}
    \right)
    \geq 
    \frac{1}{2}
    \exp\left(
        -\frac{2 \gamma \log n}{\delta^2}
    \right)
    = 
    \frac{1}{2}
    n^{-2 \gamma / \delta^2}.
\end{align*}

From Lemma \ref{lemma:general-uniform-convergence} applied to $h=f_{\sigma}$ and $h=g_{\sigma}$, we can find $E_2$ with $\prob(E_2) \leq 2e^{-c\varepsilon^2 n}$, such that, on $E_2^c$,
\begin{align*}
    &
    \sup_{\substack{
        \sigma \in [\delta, \sqrt{n}] \\
        i = 1, \dots, n
        }
    }
    \left\vert 
        \int f_{\sigma}(X_{i},x') \, d\mu_n(x')
        -
        \int f_{\sigma}(X_{i},x') \, d\mu(x')
    \right\vert
    \\
    & \quad
    +
    \sup_{\substack{
            \sigma \in [\delta, \sqrt{n}] \\
            i = 1, \dots, n
        }
    }
    \left\vert 
        \int g_{\sigma}(X_{i},x') \, d\mu_n(x')
        -
        \int g_{\sigma}(X_{i},x') \, d\mu(x')
    \right\vert
    \leq
     \, \varepsilon
    +
    O\left(
        \frac{\sqrt{\log n}}{\sqrt{n}}
    \right).
\end{align*}

Hence, on $E_1^c \cap E_2^c$, combining the inequalities above,
\[
    \sup_{\substack{
            \sigma \in [\delta, \sqrt{n}] \\
            i = 1, \dots, n
        }
    }
    \left\vert
        \frac{
            \int g_{\sigma}(X_i,x') \, d\mu_n(x')
        }{
            \int f_{\sigma}(X_i,x') \, d\mu_n(x')
        } 
        -
        \frac{
            \int g_{\sigma}(X_i,x') \, d\mu(x')
        }{
            \int f_{\sigma}(X_i,x') \, d\mu(x')
        } 
    \right\vert 
    \leq 
    C \varepsilon 
    n^{4\gamma / \delta^2}
    +
    O\left(
        \frac{
            \sqrt{\log n}
        }{
            n^{1/2 - 4\gamma / \delta^2} 
        }
    \right).
\]
Now, in order for the RHS to converge to zero, we need
  \[  \varepsilon n^{4\gamma / \delta^2}  \to 0, \quad 
    1/2 - 4\gamma / \delta^2  > 0.
\]
Moreover, we have $\prob(E_1 \cup E_2) \leq Cn^{-c_1\alpha} + 2e^{-c\varepsilon^2n}$.
Set $\varepsilon = n^{-\eta}$ for some $\eta >0$.
In order to apply Borel--Cantelli to get almost sure convergence, we want
\begin{align*}
    c_1 \alpha > 1, \quad
    1 - 2\eta > 0.
\end{align*}

Hence we need to satisfy the following collection of inequalities:
\begin{align*}
    \frac{4\gamma}{\delta^2} - \eta  < 0, \quad 
    \frac{4\gamma}{\delta^2} - 1/2  < 0, \quad
    c_1 \alpha > 1, \quad
    1 - 2\eta & > 0.
\end{align*}
Note that the last inequality gives $\eta < 1/2$ and hence the first inequality implies the second.
Now, recalling the definition of $\gamma$, the above is equivalent to
\[
    4\left(\alpha + \frac{1}{c_1} + \epsilon_n\right) < \eta \delta^2, \quad
    c_1 \alpha > 1, \quad 
    0 < \eta < 1/2,
\]
where we are free to choose $\alpha, \eta$, but $c_1$ and $\delta$ are fixed.

Let $\epsilon > 0$ (which may change from line to line below) and set
\begin{align*}
    \alpha = \frac{1}{c_1} + \epsilon, \quad 
    \eta = \frac{1}{2} - \epsilon
\end{align*}
Then all we require is that
\(
    \frac{8}{c_1\delta^2} + \epsilon
    <
    \frac{1}{2} - \epsilon.
\)
Equivalently, 
\[
    \frac{8}{c_1\delta^2} - \frac{1}{2} < 0.
\]
When this condition is satisfied, we have:
\[
    \sup_{\substack{
            \sigma \in [\delta, \sqrt{n}] \\
            i = 1,\dots,n}
        }
    \left\vert
        \frac{
            \int g_{\sigma}(X_i,x') \, d\mu_n(x')
        }{
            \int f_{\sigma}(X_i,x') \, d\mu_n(x')
        } 
        -
        \frac{
            \int g_{\sigma}(X_i,x') \, d\mu(x')
        }{
            \int f_{\sigma}(X_i,x') \, d\mu(x')
        } 
    \right\vert  
    =
    O\left(n^{8/c_1\delta^2 - 1/2}\sqrt{\log n}\right).
\]
\end{proof}

\begin{lemma}\label{lemma:convergence-of-log-f}
Let $c_1$ be the decay constant from Lemma \ref{lemma:sub_gaussian_norm_tail_bound} and let $\delta > 0$ be such that $c_1\delta^2 > 8$.
Then almost surely in $X_1, \dots, X_n$,
\[
    \sup_{\substack{
            \sigma \in [\delta, \sqrt{n}] \\
            i = 1, \dots, n
        }
    }
    \left\vert 
        \log\left(
            \int f_{\sigma}(X_i,x') \, d\mu_n(x')
        \right)
        -
        \log\left(
            \int f_{\sigma}(X_i,x') \, d\mu(x')
        \right)
    \right\vert
    =
    O(n^{4/c_1\delta^2 - 1/2} \sqrt{\log n}).
\]
\end{lemma}

\begin{proof}
Following the same method as Lemma \ref{lemma:convergence-of-g/f}, one can show that
\[
    \sup_{\substack{
            \sigma \in [\delta, \sqrt{n}] \\
            i=1, \dots, n
        }
    }
    \left\vert 
        \frac{
            \int f_{\sigma}(X_i,x') \, d\mu_n(x')
        }{
            \int f_{\sigma}(X_i,x') \, d\mu(x')
        }
        - 1
    \right\vert
    =
    O(n^{4/c_1\delta^2 - 1/2} \sqrt{\log n}).
\]
If $c_1 \delta^2 > 8$, then the right side of the equation above  is $O(n^{-\epsilon})$ for some $\epsilon > 0$.
Hence, there exists $C>0$  such that for all $i = 1, \dots, n$ and all $\sigma \in [\delta, \sqrt{n}]$
\[
    \log(1 - Cn^{-\epsilon})
    \leq 
    \log\left(
        \int f_{\sigma}(X_i,x') \, d\mu_n(x')
    \right)
    -
    \log\left(
        \int f_{\sigma}(X_i,x') \, d\mu(x')
    \right)
    \leq 
    \log(1 + Cn^{-\epsilon}),
\]
which proves the Lemma.
%
\end{proof}

We are now ready to prove Proposition \ref{prop_convergence_of_F}.

\begin{proof}[Proof of Proposition \ref{prop_convergence_of_F}]
Recalling the definition of $f_{\sigma}$ and $g_{\sigma}$ from \eqref{fandgsigma}, we have
\[
    F_{\mu}(x,\sigma)
    =
    -\frac{
        \int g_{\sigma}(x,x') \, d\mu(x')
    }{
        \int f_{\sigma}(x,x') \, d\mu(x')
    }
    -
    \log\left(
        \int f_{\sigma}(x,x') \, d\mu(x')
    \right) + \log \rho.
\]
Since $c_1 \delta^2 > 16$, by Lemma \ref{lemma:convergence-of-g/f} and Lemma \ref{lemma:convergence-of-log-f} there exists $\epsilon > 0$ such that
\begin{align*}
    \sup_{\substack{
        \sigma \in [\delta, \sqrt{n}] \\
        i=1, \dots, n
        }
    }  
    \left\vert 
        \frac{
            \int g_{\sigma}(X_i,x') \, d\mu_n(x')
        }{
            \int f_{\sigma}(X_i,x') \, d\mu_n(x')
        }
        -
        \frac{
            \int g_{\sigma}(X_i,x') \, d\mu(x')
        }{
            \int f_{\sigma}(X_i,x') \, d\mu(x')
        }
    \right\vert
    & =
    O(n^{-\epsilon}),
    \\
    \sup_{\substack{
            \sigma \in [\delta, \sqrt{n}] \\
            i = 1, \dots, n
        }
    }  
    \left\vert 
        \log\left(
            \int f_{\sigma}(X_i,x') \, d\mu_n(x')
        \right)
        -
        \log\left(
            \int f_{\sigma}(X_i,x') \, d\mu(x')
        \right)
    \right\vert
    & = 
    O(n^{-\epsilon}).
\end{align*}
Hence it follows that
\begin{align*}
     & 
     \sup_{\substack{
            \sigma \in [\delta, \sqrt{n}] \\
            i=1, \dots, n
        }
    }
    \left\vert
        F_{\mu_n}(X_i,\sigma)
        -
        F_{\mu}(X_i,\sigma)
    \right\vert
    =
    O(n^{-\epsilon}).
\end{align*}

\end{proof}

\subsection{Technical Bounds}\label{subsec:technical-bounds}

Let us collect the most important technical details from the above proof.
These will be needed for future estimates.

There exists $\gamma = 2 / c_1 + \epsilon$ such that the following hold almost surely:
\begin{enumerate}
    \item\label{a.s.-upper-bound-on-max-norm-X}
    \(
        \max_{i=1, \dots n}
        \lVert X_i \rVert 
        \leq 
        \sqrt{\gamma \log n}
    \)
    
    \item\label{a.s.-upper-bound-on-max-dist-X_i-X_j}
    \(
        \max_{i,j}
        \lVert X_i - X_j \rVert^2
        \leq 
        4 \gamma \log n
    \)
    
    \item\label{a.s.-lower-bound-int-f}
    For $\sigma \in [\delta, \sqrt{n}]$
    \[
        \int 
            \exp\left(
                -\lVert X_i - x' \rVert^2 / 2\sigma^2
            \right)
        \, d\mu_n(x'),
        \int 
            \exp\left(
                -\lVert X_i - x' \rVert^2 / 2\sigma^2
            \right)
        \, d\mu(x')        
        \geq 
        Cn^{-2\gamma / \delta^2}
    \]
\end{enumerate}

\subsection{Lower bound on $\sigma^*_{\rho , \mu}$}

Recall from Section \ref{subsec:uniqueness-of-sigma} that  $\sigma^*_{\rho, \mu}(x)$ is the unique solution of the equation $F_{\rho, \mu}(x, \sigma^*_{\rho, \mu}(x)) = 0$.
Throughout this section, let $X_1, \dots, X_n$ be i.i.d. drawn from $\mu$ and let $\mu_n = \frac{1}{n} \sum \delta_{X_i}$.
To lighten notation, let $\sigma^*(x) = \sigma^*_{\rho, \mu}(x)$ and $\sigma_n^*(x) = \sigma^*_{\rho, \mu_n}(x)$.

\begin{lemma}\label{lemma:sigma_uniform_lower_bound}
If $\mu$ satisfies the conditions of Lemma \ref{lemma_F_is_unbounded} then $\sigma^*(X_i)$ and $\sigma^*_n(X_i)$ are uniformly bounded below (in both $i$ and $n$) almost surely.
\end{lemma}

\begin{proof}
First let us consider $\sigma^*(x)$.
Let $\varepsilon > 0$ and let $f(x)$ be the density of $\mu$.
By Lemma \ref{lemma_F_is_unbounded}, we have $F_{\rho, \mu}(x, \sigma) > \varepsilon$ if $-C \log \sigma - \log f(x) + \log \rho > \varepsilon$.
Hence for some constant $c > 0$, $F_{\rho, \mu}(x,\sigma) > \varepsilon$ if 
\[
    \sigma 
    < 
    \rho^{c} e^{-c\varepsilon} f(x)^{-c}
    =: 
    \alpha(x).
\]
In particular, since $F_{\rho, \mu}(x, \sigma^*(x)) = 0$, we must have $\sigma^*(x) \geq \alpha(x)$.
Moreover, since $f(x)$ is bounded above, $\alpha(x) \geq D$ for some constant $D > 0$.

To extend the bound to $\sigma^*_n(x)$, we use Proposition \ref{prop_convergence_of_F}.
For $n$ sufficiently large,
\[
    \lvert 
    F_{\rho, \mu_n}(x, D)
    -
    F_{\rho, \mu}(x, D)
    \rvert 
    \leq 
    \varepsilon/2.
\]
In particular,
\(
    F_{\rho, \mu_n}(x, D)
    \geq 
    \varepsilon / 2.
\)
So, since $F_{\rho, \mu_n}(x,\sigma)$ is decreasing in $\sigma$, we must have $\sigma^*_n(x) \geq D$.
\end{proof}

\subsection{Uniform convergence of $\sigma^*_{\rho , \mu_n}$}

\begin{prop}\label{prop:sigma-uniform-convergence}
Suppose $\mu$ has compact support and a $C^{1}$ density.
Then almost surely, as $n \to \infty$:
\[
    \sup_{i=1, \dots, n}
    \left\vert
        \sigma^*_{\mu_n}(X_i)
        -
        \sigma^*_{\mu}(X_i)
    \right\vert
    \to 
    0.
\]
\end{prop}

\begin{proof}
By the Mean Value Theorem, there exists $\tilde{\sigma}_n(x)$ such that
\[
    \lvert 
        \sigma_n(x) - \sigma(x)
    \rvert
    =
    \frac{
        \lvert 
            F_{\mu}(x, \sigma(x)) - F_{\mu}(x, \sigma_n(x))
        \rvert
    }{
        \lvert
            \partial_{\sigma} F_{}(x, \tilde{\sigma}_n(x))
        \rvert
    }.
\]  
Moreover, since $F_{\mu}(x, \sigma(x)) = F_{\mu_n}(x, \sigma_n(x)) = 0$,
\[
    \lvert 
        \sigma_n(x) - \sigma(x)
    \rvert
    =
    \frac{
        \lvert 
            F_{\mu_n}(x, \sigma_n(x)) - F_{\mu}(x, \sigma_n(x))
        \rvert
    }{
        \lvert
            \partial_{\sigma} F(x, \tilde{\sigma}_n(x))
        \rvert
    }.
\]

Now, if $\mu$ has compact support, then it follows from the proof of Proposition \ref{prop:uniqueness_of_sigma} that $\lvert \partial_{\sigma} F(x, \tilde{\sigma}_n(x)) \rvert$ is uniformly bounded below by some positive constant.
This gives 
\[
    \lvert 
        \sigma_n(x) - \sigma(x)
    \rvert
    =
    O\left(
        \lvert 
            F_{\mu_n}(x, \sigma_n(x)) - F_{\mu}(x, \sigma_n(x))
        \rvert
    \right).
\]
Moreover, since $\mu$ has compact support, a similar argument as in the proof of Lemma \ref{lemma:sigma_uniform_lower_bound} shows that $\sigma_n(x)$ is uniformly bounded above. Hence Proposition \ref{prop_convergence_of_F}  finishes the proof.


%
\end{proof}

\section{Proof of Theorem \ref{thm:precise_theorem}}

%

\subsection{Existence of $\mu^*$}\label{sec:existencemin}

\begin{lemma}\label{lemma:P_X-is-tight}
Any sequence in $\widetilde\calP_X$ is tight.
\end{lemma}

\begin{proof}
Let $\{\mu_n\}$ be a sequence in $\widetilde\calP_X$ and let $\eps > 0$. 
We need to find $M$ such that $\forall n \geq 1$
\[
	\mu_n(\lVert X \rVert_{\infty} \geq M, \lVert Y \rVert_{\infty} \geq M) \leq \eps.
\]
We have
\begin{align*}
	\mu_n(\lVert X \rVert_{\infty}  \geq M, \lVert Y \rVert_{\infty} \geq M)
	\leq 
	\mu_n(\lVert X \rVert_{\infty} \geq M)
	=
	\mu_{X}(\lVert X \rVert_{\infty} \geq M)
	\leq 
	\varepsilon
\end{align*}
for $M$ sufficiently large.
\end{proof}

\begin{lemma}\label{lemma:existence-of-mu*-on-P_X}
There exists $\mu^*$ such that $I_{\rho}(\mu^*) = \inf_{\mu \in \widetilde\calP_X}I_{\rho}(\mu)$.
\end{lemma}

\begin{proof}
Let $\mu_n$ be such that
\[
    \inf_{\mu \in \widetilde\calP_X} I_{\rho}(\mu)
    \leq
    I_{\rho}(\mu_n) 
    \leq 
    \inf_{\mu \in \widetilde\calP_X} I_{\rho}(\mu) + \frac{1}{n}.
\]
By Lemma \ref{lemma:P_X-is-tight}, and Prokhorov's theorem, we can pass to a sub-sequence: there exists $\mu^* \in \widetilde\calP_X$ such that $\mu_{n(k)} \rightharpoonup \mu^*$.
The Lemma follows from the choice of $\mu_n$.
\end{proof}
%
%
%
%

\subsection{Convergence Lemmas for $p_{ij}$}\label{subsec:convergence-lemmas-for-p}

Define 
\begin{align*}
    \sigma_n(x) 
    & = 
    \sigma^*_{\rho, {\mu}_n}(x), 
    \\
     \sigma(x) 
    & = 
    \sigma^*_{\rho, \mu_X}(x)
    \\
    f_n(x,x') 
    & = f_{\sigma_{n}}(x,x')
    \\
      f(x,x')
      &= f_{\sigma}(x,x')
     \\
    p'_n(x,x')
    & = 
    \frac{
        f_n(x,x')
    }{
        \int f_n(x,x') \, d\mu_n(x') - \frac{1}{n}
    }, \quad
    p'(x,x')
     = 
    \frac{
        f(x,x')
    }{
        \int f(x,x') \, d\mu(x')
    }
    \\
    p_n(x,x')
    & = 
    \frac{1}{2}(p_n'(x,x') + p_n'(x',x)), \quad 
    p(x,x')
     = 
    \frac{1}{2}(p'(x,x') + p'(x',x))
\end{align*}
Note that
\(
    p_{j|i} = \frac{1}{n} p_n'(X_i, X_j), 
    p_{ij} = \frac{1}{n^2} p_n(X_i, X_j).
\)

\begin{lemma}\label{lemma:exponential-uniform-convergence} 
As $n \to \infty$:
\(
    \sup_{i,j}
    \left\vert
        f_n(X_i, X_j) - f(X_i, X_j)
    \right\vert
    =
    O(n^{-\epsilon})
\)
for some $\epsilon > 0$.
\end{lemma}

\begin{proof}
First note that by Lemma \ref{lemma:sigma_uniform_lower_bound}, $\sigma_n(x), \sigma(x)$ are uniformly bounded below and hence 
\begin{align*}
    \frac{ 
        \lvert \sigma_n(x)^2 - \sigma(x)^2 \rvert
        \, \lVert x - x' \rVert^2
    }{
        2\sigma_n(x)^2 \sigma(x)^2
    }
    & \leq 
    \frac{1}{\lvert \sigma_n(x) \rvert \lvert \sigma(x) \rvert}
    \left(
        \frac{1}{\lvert \sigma_n(x) \rvert}
        +
        \frac{1}{\lvert \sigma(x) \rvert}
    \right)
    \lvert 
        \sigma_n(x) - \sigma(x)
    \rvert 
    \lVert x - x' \rVert^2   
    \\
    & = 
    O\left(
        \lvert 
            \sigma_n(x) - \sigma(x)
        \rvert 
        \lVert x - x' \rVert^2 
    \right).
\end{align*}
If $\mu_X$ has compact support, the RHS is $o(1)$ because $\lVert x - x' \rVert^2$ is bounded and $\lvert \sigma_n(x) - \sigma(x) \rvert =O(n^{-\eps})$ for some $\eps > 0$ by Proposition \ref{prop:sigma-uniform-convergence}.
If $\mu_X$ is sub-Gaussian, then $\lVert x - x' \rVert^2$ is $O(\log n)$ almost surely and the result still holds.
Hence for $x = X_i, x' = X_j$:
\begin{align*}
     \left\vert 
        \exp(-\lVert x - x' \rVert^2 / 2 \sigma_n(x)^2) 
        - \right.&\left.
        \exp(-\lVert x - x' \rVert^2 / 2 \sigma(x)^2)
    \right\vert
    \\
     = &
   \left\vert 
        \exp(-(\sigma_n(x)^2 - \sigma(x)^2) \lVert x - x' \rVert^2 / 2 \sigma_n(x)^2 \sigma(x)^2) 
        - 
        1
    \right\vert
    \\
    =
    & O\left(
        \frac{ 
            \lvert \sigma_n(x)^2 - \sigma(x)^2 \rvert
            \, \lVert x - x' \rVert^2
        }{
            2\sigma_n(x)^2 \sigma(x)^2
        }
    \right)
    \\
    =
    & O(n^{-\epsilon}) 
\end{align*}
for some $\epsilon >0$.
\end{proof}

\begin{lemma}\label{lemma:lower-bound-p-comparison} 
If $c_1\delta^2 >32$ then there exists $\epsilon > 0$ such that as $n\to \infty$ we have
\begin{align*}
    \sup_{i, j}
    \left\vert 
        p_n'(X_i, X_j) - p'(X_i, X_j)
    \right\vert
    & =
    O\left(n^{-\eps}\right),
    \\
    \sup_{i, j}
    \left\vert 
        p_n(X_i, X_j) - p(X_i, X_j)
    \right\vert
    & =
    O\left(n^{-\eps}\right).
\end{align*}
\end{lemma}

\begin{proof}
Let us drop $X_i$ and $X_j$ from the notation, that is, write $f$ for $f(X_i, X_j)$ etc.
Then we have
\begin{align*}
    &
    \left\vert 
        p_n' - p'
    \right\vert
    =
    \left\vert
        \frac{f_n}{\int f_n \, d\mu_n - \frac{1}{n}}
        -
        \frac{f}{\int f \, d\mu}
    \right\vert
    \\
    \leq 
    &
    \left\vert
        \frac{f_n}{\int f_n \, d\mu_n - \frac{1}{n}}
        -
        \frac{f_n}{\int f_n \, d\mu_n}
    \right\vert
    +
    \left\vert
        \frac{f_n}{\int f_n \, d\mu_n}
        -
        \frac{f_n}{\int f \, d\mu_n}
    \right\vert
    +
    \left\vert
        \frac{f_n}{\int f \, d\mu_n}
        -
        \frac{f_n}{\int f \, d\mu}
    \right\vert
    +
    \left\vert
        \frac{f_n}{\int f \, d\mu}
        -
        \frac{f}{\int f \, d\mu}
    \right\vert.
\end{align*}

We bound the first term using Bound (\ref{a.s.-lower-bound-int-f}): $\int f_n \, d\mu_n \geq n^{-4/c_1\delta^2}$.
This gives
\begin{equation}\label{eqn:p-convergence-lemma-bound-1}
    \left\vert
        \frac{f_n}{\int f_n \, d\mu_n - \frac{1}{n}}
        -
        \frac{f_n}{\int f_n \, d\mu_n}
    \right\vert
    \leq 
    \frac{1}{n}
    \frac{1}{
        \left\vert
            \int f_n \, d\mu_n - \frac{1}{n}
        \right\vert
    }
    \frac{1}{
        \left\vert
            \int f_n \, d\mu_n
        \right\vert
    }
    =
    O\left(
        n^{8/c_1\delta^2 - 1}
    \right).
\end{equation}

We bound the second term using Lemma \ref{lemma:exponential-uniform-convergence} and Bound (\ref{a.s.-lower-bound-int-f}) from Section \ref{subsec:technical-bounds}:
\begin{align*}
    \left\vert
        \frac{f_n}{\int f_n \, d\mu_n}
        -
        \frac{f_n}{\int f \, d\mu_n}
    \right\vert
    & \leq 
    \frac{
        \lvert f_n - f \rvert
    }{
        \left\vert \int f_n \, d\mu_n \right\vert 
        \left\vert \int f \, d\mu_n \right\vert
    }
    =
    O\left(
        n^{16/c_1\delta^2 - 1/2} (\log n)^k
    \right),
\end{align*} which is $o(1)$ given our choice of $c_1$.

To bound the third term, using the proof of  Lemma \ref{lemma:convergence-of-g/f} and Bound (\ref{a.s.-lower-bound-int-f}) again:
\begin{align*}
    \left\vert
        \frac{f_n}{\int f \, d\mu_n}
        -
        \frac{f_n}{\int f \, d\mu}
    \right\vert
    & \leq 
    \frac{
        \left\vert
            \int f \, d\mu_n - \int f \, d\mu
        \right\vert
    }{
        \left\vert \int f \, d\mu_n \right\vert
        \left\vert \int f \, d\mu \right\vert
    }
    = 
    O\left(
        n^{16/c_1\delta^2 - 1/2}   
    \right).
\end{align*}

Finally to bound the last term use Lemma \ref{lemma:exponential-uniform-convergence} and Bound (\ref{a.s.-lower-bound-int-f}):
\[
    \left\vert
        \frac{f_n}{\int f \, d\mu}
        -
        \frac{f}{\int f \, d\mu}
    \right\vert
    =
    O\left(
        n^{12/c_1\delta^2 - 1/2}      
    \right).
\]
\end{proof}

\begin{lemma}\label{lemma:lower-bound-plogp-convergence}
If $c_1 \delta^2 > 40$ then there exists $\eps > 0$ such that as $n\to \infty$ we have
\[
    \sup_{i,j}
    \left\vert 
        p_n(X_i, X_j) \log p_n(X_i, X_j)
        -
        p(X_i, X_j) \log p(X_i, X_j)
    \right\vert
    =
    O(n^{-\eps}).
\]
\end{lemma}

\begin{proof}
Let 
\(
    \epsilon_n
    =
    n^{16/c_1\delta^2 - 1/2} (\log n)^k, 
\)
for some $k >0$ and for simplicity, write $p = p(X_i,X_j), p_n = p_n(X_i,X_j)$.
By Lemma \ref{lemma:lower-bound-p-comparison} we have
\begin{align*}
	p_n \log p_n
    & = 
    p \log p
    + 
    p \log\left(\frac{p_n}{p}\right)
    +
    (p_n - p) \log p_n
    \\
    & = 
    p \log p 
    +
    p \log\left(\frac{p + O(\epsilon_n)}{p}\right)
    +
    O(\epsilon_n) \log p_n.
\end{align*}

Let us bound the second term.
Using Bound (\ref{a.s.-lower-bound-int-f}) of Section \ref{subsec:technical-bounds}, $p \geq Cn^{-4 / c_1 \delta^2}$.
Hence 
\[
	\frac{p + O(\epsilon_n)}{p}
	=
	1 + \frac{O(\epsilon_n)}{p}
	=
	1 + O(\epsilon_n n^{4/c_1 \delta^2}).	
\]
This is $o(1)$ given our choice of $c_1$. In particular, the second term is positive for large $n$.
We can give an upper bound using $\log(1 + x) \leq x$ to see that
\[
	p \log\left(\frac{p + O(\epsilon_n)}{p}\right)
	\leq 
	p \frac{O(\epsilon_n)}{p}
	=
	O(\epsilon_n).	
\]

To show the third term is $o(1)$, using Bound (\ref{a.s.-lower-bound-int-f}), we have $p_n \leq Cn^{4 / c_1\delta^2}$.
Hence
\[
    O(\eps_n) \log p_n
    =
    O(\eps_n \log n)
\]
which again is $o(1)$ given our choice of $c_{n}$.
\end{proof}


Lastly let us show that whether we include diagonal terms or not in the computation of $\sigma_i$ has no effect in the large $n$ limit.

\begin{lemma}\label{lemma:sigma_i-well-defined-large-n}
If $c_1\delta^2 > 8$ then as $n\to \infty$ we have
\[
    \sum_j
    p_{j \vert i} \log p_{j \vert i}
    +
    \log(\rho n)
    =
    o(1).
\]
\end{lemma}

\begin{proof}
Note that $p_{j|i} = \frac{1}{n} p_n'(X_i, X_j)$.
Hence
\begin{align*}
    \sum_{j \neq i} 
        p_{j|i} \log p_{j|i} + \log(\rho n) 
	& = 
   \sum_{j \neq i} 
       	\frac{1}{n}
        p_n'(X_i, X_j)
        \log\left(
            \frac{1}{n}
            p_n'(X_i, X_j)  
        \right)
    + \log n + \log \rho
    \\
    & = 
    \frac{1}{n}
    \sum_{j \neq i} 
        p_n'(X_i, X_j) \log p_n'(X_i, X_j)
    + 
    \log \rho 
    +
    \log n
    \left(
        1 - \frac{1}{n} \sum_{j \neq i} p_n'(X_i, X_j)
    \right).
\end{align*}
It follows from the definition of $p'_n$ that $1 - \frac{1}{n} \sum_{j \neq i} p_n'(X_i, X_j) = 0$.
Next we need to replace sum of $p'_n \log p'_n$ terms with an integral over $\mu_n$.
The resulting difference is given by the diagonal term:
\begin{align*}
	& \frac{1}{n}
	\sum_{j \neq i} 
   	p_n'(X_i, X_j) \log p_n'(X_i, X_j)
	-
	\int
        p_n'(X_i, x') \log p_n'(X_i, x') 
    \, d\mu_n(x')		
    \\
    =
    & -\frac{1}{n}
    \frac{1}{
       \int f(X_i, x') \, d\mu_n(x') - \frac{1}{n}
    }
    \log\left(
        \frac{1}{
           \int f(X_i, x') \, d\mu_n(x') - \frac{1}{n}
        }
    \right).
\end{align*}
Let us bound this diagonal term.
If $\mu_X$ has compact support then $\exp(-\lVert X_i - x' \rVert^2 / 2 \sigma_i^2)$ is bounded below (uniformly in $i$ and $n$) and so this term is $O(n^{-1})$.
If $\mu_X$ has sub-Gaussian tails then from Bound (\ref{a.s.-lower-bound-int-f}) of Section \ref{subsec:technical-bounds}, almost surely we have
\[
    \int 
        \exp(-\lVert X_i - x' \rVert^2 / 2 \sigma_i^2)
    \, d\mu_n(x')
    - 
    \frac{1}{n}
    \geq 
    \frac{1}{2}
    n^{-4 / c_1\delta^2}
\]
and hence the diagonal term is 
\(
    O\left(
        n^{4/c_1\delta^2 - 1} \log n
    \right)
\)
which is $o(1)$ for our choice of $c_1$.

Finally we have
\begin{align*}
    \int
        p_n'(X_i, x') \log p_n'(X_i, x') 
    \, d\mu_n(x')
    + 
    \log \rho
    =
    F_{\rho, \mu_n}(X_i, \sigma_n(X_i))
    +
    O\left(
        n^{8/c_1\delta^2 - 1}
    \right).    
\end{align*}
This follows from \eqref{eqn:p-convergence-lemma-bound-1} in the proof in Lemma \ref{lemma:lower-bound-p-comparison} and by the proof of Lemma \ref{lemma:lower-bound-plogp-convergence}. The Lemma follows.
\end{proof}

\subsection{Lower bound for $\inf_{Y \in (\mathbb R^s)^n} L_{n,\rho}(X, Y)$}


\begin{lemma}\label{minimizer of L_n(X,y) exists}
There exists a set of points $y^*(X)$ such that $\inf_{Y \in (\mathbb R^s)^n} L_{n,\rho}(X, Y) = L_{n,\rho}(X, y^*)$.
\end{lemma}

\begin{proof}
Let $k \in \{1, \dots, n\}$.
We are going to show that
\[
    \lim_{\lvert y_k \rvert \to \infty} L_{n, \rho}(X,y) = \infty.
\]

It suffices to show that
\begin{align*}
 \sum_{i \neq j}p_{ij}        \log \left( 
            \frac{
                (1 + \lVert y_i - y_j \rVert^2)^{-1}
            }{
                \frac{1}{n^2} \sum_{m \neq \ell} (1 + \lVert y_m - y_\ell \rVert^2)^{-1}
            } 
        \right) 
    \to -\infty
\end{align*}

as $\lvert y_k \rvert \to \infty$.

First note that, as $\lvert y_k \rvert \to \infty$, 
\begin{align*}
     \quad
    -\sum_{i \neq j} 
        p_{ij}
        \log \left( 
            \frac{1}{n^2} \sum_{\ell \neq m} (1 + \lVert y_\ell - y_m \rVert^2)^{-1}
        \right)
      & = 
    - \log \left( 
        \frac{1}{n^2} \sum_{\ell \neq m} (1 + \lVert y_\ell - y_m \rVert^2)^{-1}
    \right)
    \\
    & \leq
    - \log \left( 
        \frac{1}{n^2} \sum_{m \neq k, \ell \neq k} (1 + \lVert y_m - y_\ell \rVert^2)^{-1}
    \right) < \infty.
\end{align*}

Moreover, as $\lvert y_k \rvert \to \infty$, $\log \left( (1 + \lVert y_i - y_k \rVert^2)^{-1} \right) \to -\infty$ and hence:
\begin{align*}
    & \quad 
  \sum_{i\neq j} p_{ij}        \log \left( 
            (1 + \lVert y_i - y_j \rVert^2)^{-1}
        \right)
        \to -\infty.
\end{align*}
Hence, the minimum of $L_{n, \rho}(X,y)$ is achieved on a bounded set, and so, by continuity of $L_{n, \rho}(X,y)$ in $y$ the Lemma follows.

\end{proof}

Let $y^*$ be given by Lemma \ref{minimizer of L_n(X,y) exists} and define 
\begin{equation}\label{defn mu_n}
	\mu_n = \frac{1}{n} \sum \delta_{X_i} \otimes \delta_{y_i^*}.
\end{equation}

By the same proof as in Lemma \ref{lemma:P_X-is-tight} we can pass to a sub-sequence, $n(k)$, such that ${\mu}_{n(k)}$ converges weakly to some probability measure ${\mu}$ on $\R^d \times \R^s$.
In particular, note that if $f(x)$ is a bounded continuous function on $\R^d$ only, then
\begin{align*}
    \frac{1}{n(k)} \sum_i f(x_i)
    & = 
    \int f(x) \, d{\mu}_{n(k)}(x, y)
    \to 
    \int f(x) \, d{\mu},
    \\
    \frac{1}{n(k)} \sum_i f(x_i)
   & \to 
    \int f(x) \, d\mu_X.
\end{align*}
So we see that ${\mu} \in \calP_X$.
In fact, by Proposition \ref{prop:appendix-prop}, $\mu \in \widetilde\calP_X$.



\begin{lemma}\label{lemma:lower-bound-plogp-integral-convergence}
 As $n \to \infty$ we have
\begin{equation}\label{eqn:lower-bound-plogp-convergence}
    \sum_{i \neq j}
        p_{ij} \log p_{ij}
    =
    \iint p(x,x') \log p(x,x') \, d\mu d\mu - 2\log n 
    +
    o(1).
\end{equation}
\end{lemma}

\begin{proof}
First mimic the proof of Lemma \ref{lemma:sigma_i-well-defined-large-n}: by using Lemma \ref{lemma:lower-bound-plogp-convergence}, one can show that 
\[
	\sum_{i \neq j} 
		p_{ij} \log p_{ij}
	= 
	\iint 
		p(x,x') \log p(x,x') 
	\, d\mu_n d\mu _n
	-
	2 \log n
	+ 
	o(1).	
\]
Then following the same steps as in Lemma \ref{lemma:general-uniform-convergence} and Lemma \ref{lemma:convergence-of-log-f} one can show that
\[
	\iint 
		p(x,x') \log p(x,x') 
	\, d\mu_n d\mu_n
	=
		\iint 
		p(x,x') \log p(x,x') 
	\, d\mu d\mu
	+ o(1).
\]
We leave the details to the reader.
\end{proof}

\begin{lemma} \label{coffeemugonMonday}The following holds:
\[
    \liminf_{n \to \infty}
    \sum_{i \neq j}
        p_{ij} \log \left(\frac{p_{ij}}{q_{ij}}\right)
    \geq 
    \iint 
        p(x,x') \log \frac{p(x,x')}{q(y,y')}
    \, d{\mu} d{\mu}.
\]
\end{lemma}

\begin{proof}
Convergence of $\sum_{i \neq j} p_{ij} \log p_{ij}$ has been dealt with in Lemma \ref{lemma:lower-bound-plogp-integral-convergence}.
To deal with $\sum_{i \neq j} p_{ij}\log q_{ij}$ term, we have
\begin{align}\label{eqn:lower-bound-plogq-convergence}
	\nonumber
    & - \sum_{i \neq j} p_{ij} \log q_{ij} 
    \\
    \nonumber
    =
    & \frac{1}{n^2}
    \sum_{i \neq j}
        p_n(X_i, X_j)
        \log(1 + \lVert y_i^* - y_j^* \rVert^2)
    +
    \log\left(
    	  \frac{1}{n^2}
        \sum_{i \neq j}
            (1 + \lVert y_i^* - y_j^* \rVert^2)^{-1}
    \right)
    +
    2 \log n
    \\
    = 
    & \iint 
        p_n(x,x') \log(1 + \lVert y - y' \rVert^2)
    \, d{\mu}_n d{\mu}_n
    +
    \log\left(
        \iint 
            (1 + \lVert y - y' \rVert^2)^{-1}
        \, d{\mu}_n d{\mu}_n
        -
        \frac{1}{n}
    \right)
    +
    2 \log n.
\end{align}
By Fatou's Lemma and Lemma \ref{lemma:lower-bound-p-comparison}, we have
\begin{align*}
    \liminf_{n \to \infty}
    \iint 
        p_n(x,x') \log(1 + \lVert y - y' \rVert^2)
    \, d{\mu}_n d{\mu}_n   
    & \geq 
    \iint 
        p(x,x') \log(1 + \lVert y - y' \rVert^2)
    \, d{\mu} d{\mu}.
\end{align*}
Similarly, using Fatou again:
\[
    \liminf_{n \to \infty}
    \iint 
        (1 + \lVert y - y' \rVert^2)^{-1}
    \, d{\mu}_n d{\mu}_n
    -
    \frac{1}{n}
    \geq 
    \iint 
        (1 + \lVert y - y' \rVert^2)^{-1}
    \, d{\mu} d{\mu}.
\]
Then, since $\log$ is continuous and increasing
\begin{align*}
    \liminf_{n \to \infty}
    \log\left(
        \iint 
            (1 + \lVert y - y' \rVert^2)^{-1}
        \, d{\mu}_n d{\mu}_n
        -
        \frac{1}{n}
    \right)
    \geq 
    \log\left(
        \iint 
            (1 + \lVert y - y' \rVert^2)^{-1}
        \, d{\mu} d{\mu}  
    \right).
\end{align*}
The $-2\log n$ term in \eqref{eqn:lower-bound-plogp-convergence} cancels with the $2\log n$ term in \eqref{eqn:lower-bound-plogq-convergence}, which completes the proof of the Lemma.
\end{proof}

To summarize, combining Lemmas \ref{minimizer of L_n(X,y) exists} and \ref{coffeemugonMonday}   we have proved the following Proposition:

\begin{prop}\label{prop-liminf-d_n-geq-I(lambda)}
There exists a sub-sequence $n(k)$ such that
\[
	\liminf_{k \to \infty} \inf_{Y \in (\mathbb R^s)^{n(k)}} L_{n(k),\rho}(X, Y) 
	\geq 
	I_{\rho}({\mu})
\]
for some ${\mu} \in \widetilde\calP_X$.
\end{prop}

\subsection{Upper bound for $\inf_{Y \in (\mathbb R^s)^{n}} L_{n,\rho}(X, Y)$}

From Lemma \ref{lemma:existence-of-mu*-on-P_X}, there exists $\mu^*$ such that $I(\mu^*) = \inf_{\mu} I_{\rho}(\mu)$.
Let $(X_1, Y_1), \dots, (X_n, Y_n)$ be i.i.d. random variables drawn from $\mu^*$ and consider the measures $\mu_n = \frac{1}{n} \sum_{i=1}^n \delta_{(X_i,Y_i)}$.
Then we know that $\mu_n \rightharpoonup \mu^*$.


Define $\sigma, \sigma_n, f, f_n, p$ and $p_n$ as in Section \ref{subsec:convergence-lemmas-for-p} and let
\begin{align*}
    g(y,y')
    & = 
    (1 + \lVert y - y' \rVert^2)^{-1},
    \\
    q_n(y,y')
    & = 
    \frac{
        g(y,y')
    }{
        \iint g(y,y') \, d\mu_n(y) d\mu_n(y') - \frac{1}{n}
    }.
\end{align*}

\begin{lemma}\label{lemma:discrete-entropy-convergence-with-error-term}
As $n \to \infty$, we have
\begin{equation}\label{eqn:pen}
    \sum_{i \neq j} p_{ij}     
        \log\left(
            \frac{p_{ij}}{q_{ij}}
        \right)
    =
    \iint 
        p_n(x,x')
        \log\left(
            \frac{p_n(x,x')}{q_n(y,y')}
        \right)
    \, d\mu_n d\mu_n
    +
    o(1).
\end{equation}
\end{lemma}

\begin{proof}
We have 
\begin{align*}
    \sum_{i \neq j} p_{ij}     
        \log\left(
            \frac{p_{ij}}{q_{ij}}
        \right)
	 & = 
	 \frac{1}{n^2}
	 \sum_{i \neq j}
	     p_n(X_i,X_j) \log\left(\frac{p_n(X_i,X_j)}{q_n(X_i,X_j)}\right)
    \\
    & = 
    \iint 
        p_n(x,x')
        \log\left(
            \frac{p_n(x,x')}{q_n(y,y')}
        \right)
    \, d\mu_n d\mu_n
    +
    \mathrm{Err}_n,
    \\
    \mathrm{Err}_n
    & := 
	 -
	 \frac{1}{n^2}
	 \sum_{i=1}^n
	     \frac{
            1
        }{
            \frac{1}{n}
            \sum_{k \neq i} f_n(X_i, X_k)
        }
        \left[
            \log\left(
                \frac{
                    1
                }{
                    \frac{1}{n}
                    \sum_{k \neq i} f_n(X_i, X_k)
                }        
            \right)
            -
            \log\left(
                \frac{
                    1
                }{
                    \frac{1}{n^2}
                    \sum_{k \neq l} g(Y_k, Y_l)
                }        
            \right)
        \right].
\end{align*}

Now let us show that $\mathrm{Err}_n = o(1)$.
Since $f(x,x') \in [0,1]$ we have
\(
    \frac{1}{n}
    \sum_{k \neq i}
        f(X_i, X_k)
    \in [0,1].
\)
Moreover, $x \log x$ is bounded on $[0,1]$, so
\[
    \mathrm{Err}_n
    =
    O(n^{-1})
    +
    \frac{1}{n^2}
    \sum_{i=1}^n
        \frac{
        1
        }{
            \frac{1}{n}
            \sum_{k \neq i} f_n(X_i, X_k)
        }
        \log\left(
            \frac{
                1
            }{
                \frac{1}{n^2}
                \sum_{k \neq l} g(Y_k, Y_l)
            }
        \right).            
\]
Now, using Bound (\ref{a.s.-lower-bound-int-f}) of Section \ref{subsec:technical-bounds}, we have 
\[
    \frac{1}{n^2}
    \sum_{i=1}^n
    \frac{
        1
    }{
        \frac{1}{n}
        \sum_{k \neq i}
            f_n(X_i, X_k)        
    }
    =
    O(n^{4/c_1\delta^2 - 1})
\]
which is $o(1)$ for $c_1$ sufficiently large.
Moreover, by Fatou
\[
    \liminf_{n \to \infty}
    \frac{1}{n^2}
    \sum_{k \neq l}
        g(Y_k, Y_l)
    =
    \liminf_{n \to \infty}
    \iint g(y,y') \, d\mu_n d\mu_n
    -
    \frac{1}{n}
    \geq 
    \iint g(y,y') \, d\mu^* d\mu^*
    > 0.
\]
The Lemma follows.
\end{proof}

We can apply same proof as Lemma \ref{lemma:lower-bound-plogp-integral-convergence} to show convergence of the $\iint p_n \log p_n \, d\mu_n d\mu_n$ term in \eqref{eqn:pen}.
The main difference between the proof of the upper bound and the proof of the lower bound comes in dealing with the $\iint p_n \log q_n \, d\mu_n d\mu_n$ term.
Up to a factor of $\log n$ (which cancels with an analogous factor in the $p_n \log p_n$ term), we have
\[
	-
	\iint p_n \log q_n \, d\mu_n d\mu_n
	=
	\iint 
		p_n \log(1 + \lVert y - y' \rVert^2)
	\, d\mu_n d\mu_n
	+
	\log\left(
		\iint (1 + \lVert y - y' \rVert^2)^{-1} \, d\mu_n d\mu_n
		-
		\frac{1}{n}
	\right).
\]

\begin{lemma}\label{lemma:Z_q_n-convergence}
As $n \to \infty$, we have
\[
    \iint (1 + \lVert y - y' \rVert^2)^{-1} \, d\mu_n(y) d\mu_n(y') 
    -
    \frac{1}{n}
    =
    \iint (1 + \lVert y - y' \rVert^2)^{-1} \, d\mu^* d\mu^*
    +
    o(1).
\]
\end{lemma}

\begin{proof}
Since $\mu_n \rightharpoonup \mu^*$, the sequence $\{\mu_n\}$ is tight: for all $\varepsilon > 0$, there exists compact $K$ such that for all $n$, 
\(
    \mu^*(K), \mu_n(K) \leq \varepsilon.
\)
Hence, since $(1 + \lVert y - y' \rVert^2)^{-1}$ is bounded, there exists a constant $C > 0$ such that
\begin{align*}
    & \left\vert
        \iint (1 + \lVert y - y' \rVert^2)^{-1} \, d\mu_n d\mu_n
        -
        \iint (1 + \lVert y - y' \rVert^2)^{-1} \, d\mu^* d\mu^* 
    \right\vert 
    \\
    \leq
    & \left\vert 
        \iint 
           (1 + \lVert y - y' \rVert^2)^{-1} \1_{K \times K} 
        \, (d\mu_n d\mu_n - d\mu^* d\mu^*)
    \right\vert
    +
    C \varepsilon.
\end{align*}

Define
\begin{align*}
    G_n(y) = \int (1 + \lVert y - y' \rVert^2)^{-1} \1_K \, d\mu_n(y'), 
    \quad
    G(y) = \int (1 + \lVert y - y' \rVert^2)^{-1} \1_K \, d\mu^*(y').
\end{align*}

\begin{claim}
$G_n(y)$ is uniformly equicontinuous on $K$.
\end{claim}

\begin{proof}[Proof of Claim]
We have
\[
    \left\vert
        G_n(y) - G_n(\tilde y)
    \right\vert
    =
    \int 
        \frac{
            \left\Vert y - y' \right\Vert^2
            -
            \left\Vert \tilde y - y' \right\Vert^2
        }{
            (1 + \left\Vert y - y' \right\Vert^2)
            (1 + \left\Vert \tilde y - y' \right\Vert^2)
        }
        \1_K
    \, d\mu_n(y').
\]
Now, since $K$ is compact, for $y, \tilde y \in K$, we have
\begin{align*}
    \left\vert
        \lVert y - y' \rVert^2
        -
        \lVert \tilde y - y' \rVert^2
    \right\vert
    =
    \left\vert 
        \lVert y - \tilde y \rVert^2
        -
        2\langle
            y - \tilde y, y' - \tilde y
        \rangle
    \right\vert
    =
    O(\lVert y - \tilde y \rVert).
\end{align*}
Hence, uniformly in $n$, 
\(
    \left\vert
        G_n(y) - G_n(\tilde y)
    \right\vert
    =
    O(\lVert y - \tilde y \rVert).
\)
\end{proof}

\begin{claim} \label{Danishavingcoffee}We have
    $\sup_{y \in K} 
    \left\vert G_n(y) - G(y) \right\vert
    =
    o(1).$
\end{claim}

\begin{proof}[Proof of Claim]
Given $\varepsilon > 0$ we will show that for each $y \in K$ there exists $\delta > 0$ such that
\[
    \sup_{\tilde y \in B(y,\delta)} 
    \left\vert G_n(\tilde y) - G(\tilde y) \right\vert
    \leq
    \varepsilon.    
\]
Since $K$ is compact, this implies the claim.

We have
\[
    \left\vert 
        G_n(\tilde y) - G(\tilde y) 
    \right\vert
    \leq
    \left\vert 
        G_n(\tilde y) - G_n(y)
    \right\vert
    +
    \left\vert
        G_n(y) - G(y)
    \right\vert
    +
    \left\vert 
        G(y) - G(\tilde y)
    \right\vert.
\]
By the uniform equicontinuity of $G_n$, there exists $\delta$ such that, on $B(y, \delta) >0$ the first term is bounded by $\varepsilon$  uniformly in $n$.
For $n$ sufficiently large, the second term is bounded by $\varepsilon$ since $g$ is bounded and continuous and $\mu_n \rightharpoonup \mu^*$.
Finally the last term is bounded on a ball $B(y, \delta')$ by continuity of $G$ (same proof as for equicontinuity of $G_n$).
\end{proof}

By the Claim \ref{Danishavingcoffee} we have
\[
    \iint g(y,y') \1_{K \times K} \, d\mu_n d\mu_n
    =
    \iint g(y,y') \1_{K \times K} \, d\mu^{*} d\mu_n
    +
    o(1)
    =
    \int G(y) \1_K \, d\mu_n(y) + o(1).
\]
Since $G$ is bounded and continuous, by weak convergence of $\mu_n$, we have
\[
    \int G(y) \1_K \, d\mu_n(y) 
    =
    \int G(y) \1_K \, d\mu^*(y) + o(1)
    =
    \iint g \1_{K \times K} \, d\mu^* d\mu^* + o(1).
\]
\end{proof}

\begin{lemma}\label{lemma:upper-bound-plogq-convergence}
If $\mu^*$ has compact support in $Y$, then
\[
	\iint 
		p_n(x,x') \log(1 + \lVert y - y' \rVert^2)
	\, d\mu_n d\mu_n
	=
	\iint 
		p(x,x') \log(1 + \lVert y - y' \rVert^2)
	\, d\mu^* d\mu^*
	+
	o(1).
\]
\end{lemma}

\begin{proof}
By Lemma \ref{lemma:lower-bound-plogp-convergence}, $p_n(x,x') = p(x,x') + o(1)$.
Since $\mu^*$ has compact support in $Y$, it follows that
\[
	\iint 
		p_n(x,x') \log(1 + \lVert y - y' \rVert^2)
	\, d\mu_n d\mu_n
	=
	\iint 
		p(x,x') \log(1 + \lVert y - y' \rVert^2)
	\, d\mu_n d\mu_n
	+
	o(1).
\]
Then, following the proof sketched in Lemma \ref{lemma:lower-bound-plogp-integral-convergence}, one can show that
\[
	\iint 
		p(x,x') \log(1 + \lVert y - y' \rVert^2)
	\, d\mu_n d\mu_n
	=
	\iint 
		p(x,x') \log(1 + \lVert y - y' \rVert^2)
	\, d\mu^* d\mu^*
	+
	o(1).
\]
\end{proof}

\begin{lemma}
If $\mu^*$ has compact support in $Y$ then 
\[
    \sum_{i \neq j} p_{ij}     
        \log\left(
            \frac{p_{ij}}{q_{ij}}
        \right)
	 =	
	 \iint 
	     p(x,x') \log\left(\frac{p(x,x')}{q(y,y')}\right)
	 \, d\mu^* d\mu^*
	 +
	 o(1).
\]
\end{lemma}

\begin{proof}
This follows as a combination of Lemma \ref{lemma:lower-bound-plogp-integral-convergence}, Lemma \ref{lemma:discrete-entropy-convergence-with-error-term}, Lemma \ref{lemma:Z_q_n-convergence} and Lemma \ref{lemma:upper-bound-plogq-convergence}.
\end{proof}

Thus it follows from the Lemmas above that
\begin{prop}\label{prop:limsup_dn(X)_upper_bound} 
If $\mu^*$ has compact support in $Y$ then
\[
	\limsup_{n \to \infty} \inf_{Y \in (\mathbb R^s)^{n}} L_{n,\rho}(X, Y) 
	\leq 
	I_{\rho}(\mu^*).
\]
\end{prop}

\subsection{Proof of Theorem \ref{thm:precise_theorem}}

By Theorem \ref{thm2} (whose proof is independent of Theorem \ref{thm:precise_theorem}), when $\mu_X$ has compact support, $\mu^*$ has compact support in $Y$.
In particular, when $\mu_X$ has compact support, Proposition \ref{prop:limsup_dn(X)_upper_bound} applies.

Let us write
\[
	d_{n, \rho}(X) = \inf_{Y \in (\mathbb R^s)^{n}} L_{n,\rho}(X,Y).
\]
By Proposition \ref{prop-liminf-d_n-geq-I(lambda)}, Proposition \ref{prop:limsup_dn(X)_upper_bound} and the fact that $\mu_* = \argmin_{\mu \in \widetilde\calP_X} I_{\rho}(\mu)$, there exists $\mu \in \widetilde\calP_X$ and a subsequence $n(k)$ such that
\[
    I_{\rho}(\mu_*) 
    \geq 
    \limsup_{n \to \infty} d_{n, \rho}(X)
    \geq
    \liminf_{k \to \infty}d_{n(k), \rho}(X)
    \geq 
    I_{\rho}({\mu})
    \geq 
    I_{\rho}(\mu_*).
\]
Hence $\lim_{k \to \infty} d_{n(k), \rho}(X) = I_{\rho}(\mu) = I_{\rho}(\mu_*)$ with $\frac{1}{n}\sum_i \delta_{(X_i, y^*_i)}$ converging weakly to $\mu$.

To see that $\lim_{n \to \infty} d_{n, \rho}(X) = I(\mu_*)$, let $n(m)$ be any sequence such that $d_{n(m), \rho}(X)$ converges.
If $d_{n(m), \rho}(X) \to I'$ then by the same reasoning as above, we can find a sub-sequence $n(m_k)$ such that $d_{n(m_k), \rho}(X) \to I_{\rho}(\mu_*)$ and hence $I' = I_{\rho}(\mu_*)$.

\section{Proof of Theorem \ref{thm2}}

Recall that 
\[
    \widetilde{\calP}_X
    =
    \left\{
        \mu \in \calP_X: 
        \int 
            (p(x,x') - q_{}(y,y'))
            \frac{y - y'}{1 + \lVert y - y' \rVert^2}
        \, d\mu(x',y') 
        = 0, 
        \, \mu(x,y)\text{-a.e.}
    \right\}.
\]



\begin{lemma}\label{lemma:plog(1+(y-y')^2)-decomposition}
For every $\mu \in \widetilde{\calP}_X$, there exists a function $h_{\mu}(x)$ that does depend on $y$ such that
\[
    \int 
        p(x,x') \log(1 + \lVert y - y' \rVert^2)
    \, d\mu 
    =
    h_{\mu}(x)
    -
    \frac{
        \int 
            (1 + \lVert y - y' \rVert^2)^{-1} 
        \, d\mu
    }{
        \iint 
            (1 + \lVert y - y' \rVert^2)^{-1} 
        \, d\mu d\mu
    } .
\]
Moreover, if $\mu_X$ has compact support, then $h_{\mu}$ is continuous on $\mu_X$.
\end{lemma}

\begin{proof}
We have
\begin{align*}
     0
     = &
     \int 
        (p(x,x') - q_{}(y,y'))
        \frac{y - y'}{1 + \lVert y - y' \rVert^2}
    \, d\mu(x',y')
    \\
    = &
    \int 
        p(x,x') \frac{y-y'}{1 + \lVert y - y' \rVert^2}
    \, d\mu 
    -
    \frac{  
        \int 
            \frac{
                y - y'
            }{
                (1 + \lVert y - y' \rVert^2)^2
            }
        \, d\mu
    }{
        \iint 
            \frac{1}{1 + \lVert y - y' \rVert^2} 
        \, d\mu d\mu
    }
    \\
    = &
    \frac{1}{2} \nabla_y
    \left(
        \int 
            p(x,x') \log(1 + \lVert y - y' \rVert^2)
        \, d\mu 
        +
        \frac{
            \int 
                (1 + \lVert y - y' \rVert^2)^{-1} 
            \, d\mu
        }{
            \iint 
                (1 + \lVert y - y' \rVert^2)^{-1} 
            \, d\mu d\mu
        } 
    \right).
\end{align*}

It follows that there exists $h_{\mu}:\R^d \to \R$ such that
\[
    \int 
        p(x,x') \log(1 + \lVert y - y' \rVert^2)
    \, d\mu 
    +
    \frac{
        \int 
            (1 + \lVert y - y' \rVert^2)^{-1} 
        \, d\mu
    }{
        \iint 
            (1 + \lVert y - y' \rVert^2)^{-1} 
        \, d\mu d\mu
    } 
    =
    h_{\mu}(x).
\]

To show that $h_{\mu}$ is continuous, it suffices to show that $\int p(x,x') \log(1 + \lVert y - y' \rVert^2) \, d\mu$ is continuous in $x$.
Let $\varepsilon > 0$ and suppose that $x_n \to x$.
Since $p$ is continuous, and $\mu_X$ has compact support, we have
\[
    \iint \log(1 + \lVert y - y' \rVert^2) \, d\mu d\mu
    \leq 
    C \iint 
        p(x,x') \log(1 + \lVert y - y' \rVert^2)
    \, d\mu d\mu
    <
    \infty.
\]
In particular, for almost every $y$, $\int \log(1 + \lVert y - y' \rVert^2) \, d\mu < \infty$.

Since $p(x,x')$ is continuous and $\mu_X$ has compact support, $p(x,x')$ is uniformly continuous.
In particular, there exists $\delta = \delta(y, \mu)$ such that if $\lVert x_n - x \rVert < \delta$ then 
\[
    \sup_{x'} \lvert p(x_n, x') - p(x, x') \rvert
    <
    \frac{
        \varepsilon
    }{
        \int \log(1 + \lVert y - y' \rVert^2) \, d\mu
    }.
\]
In particular,
\[
    \left\vert 
        \int 
            (p(x_n, x') - p(x,x'))
            \log(1 + \lVert y - y' \rVert^2)
        \, d\mu
    \right\vert
    < 
    \varepsilon.
\]
\end{proof}

\begin{proof}[Proof of Theorem \ref{thm2}]
Define $D_n = \{ \lVert y \rVert \geq n\}$.
If $\mu^*$ does not have compact support in $Y$ then $\mu^*(D_n) > 0$ for all $n$.

By Lemma \ref{lemma:plog(1+(y-y')^2)-decomposition}, there exists continuous $h(x)$ such that
\[
    \int p\log(1 + \lVert y - y' \rVert^2) \, d\mu^*
    =
    h(x) 
    -
    \frac{
        \int (1 + \lVert y - y' \rVert^2)^{-1} \, d\mu^*
    }{
        \iint (1 + \lVert y - y' \rVert^2)^{-1} \, d\mu^* d\mu^*
    }.
\]
In particular,
\[
    \iint 
        p \log(1 + \lVert y - y' \rVert^2) 
        \frac{\1_{D_n}(x,y)}{\mu^*(D_n)}
    \, d\mu^* d\mu^*
    =
    \int 
        h(x) \frac{\1_{D_n}(x,y)}{\mu^*(D_n)} 
    \, d\mu^*
    +
    \frac{
        \iint 
            (1 + \lVert y - y' \rVert^2)^{-1} 
            \frac{\1_{D_n}(x,y)}{\mu^*(D_n)}            
        \, d\mu^* d\mu^*
    }{
        \iint 
            (1 + \lVert y - y' \rVert^2)^{-1} 
        \, d\mu^* d\mu^*
    }.  
\]
Since $h$ is continuous, and $\mu_X$ has compact support, the first term of the RHS is bounded uniformly in $n$.
Moreover, $(1 + \lVert y - y' \rVert^2)^{-1}$ is bounded, so the second term of the RHS is also bounded uniformly in $n$.
It follows that the LHS is bounded uniformly in $n$.

Now, since $p(x,x') > 0$ and $\mu_X$ has compact support, there exists a constant $c > 0$ such that $p(x,x') \geq c$ for all $x, x'$ in the support of $\mu_X$.
Hence
\begin{align*}
    & 
    \frac{1}{\mu^*(D_n)}
    \iint 
        p \log(1 + \lVert y - y' \rVert^2) 
        \1\{\lVert y \rVert \geq n\}
    \, d\mu^* d\mu^*
    \\
    \geq &
    \frac{1}{\mu^*(D_n)}
    \iint 
        p \log(1 + \lVert y - y' \rVert^2) 
        \1\{\lVert y \rVert \geq n\}
        \1\{\lVert y - y' \rVert \geq n/2\}
    \, d\mu^* d\mu^*
    \\
    \geq & 
    c \frac{\log(1 + n^2/4)}{\mu^*(D_n)}
    \iint 
        \1\{\lVert y \rVert \geq n\}
        \1\{\lVert y' \rVert \leq n/2\}
    \, d\mu^* d\mu^*
    \\
    = & 
    c \log(1 + n^2/4) \mu^{*}(\lVert Y' \rVert \leq n/2)
    \\
    \geq &
    \frac{c}{2} \log(1 + n^2/4).
\end{align*}
Taking $n \to \infty$, we get a contradiction.
Hence it must be the case that $\mu^*(D_n) = 0$ for $n$ sufficiently large.
That is, $\mu^*(\lVert Y \rVert \geq n) = 0$, so $\mu^*$ has compact support in $Y$.
\end{proof}

\bibliographystyle{plain}
\bibliography{refs}

\begin{thebibliography}{10}

\bibitem{pmlr-v75-arora18a}
Sanjeev Arora, Wei Hu, and Pravesh~K. Kothari.
\newblock An analysis of the t-sne algorithm for data visualization.
\newblock In Sébastien Bubeck, Vianney Perchet, and Philippe Rigollet,
  editors, {\em Proceedings of the 31st Conference On Learning Theory},
  volume~75 of {\em Proceedings of Machine Learning Research}, pages
  1455--1462. PMLR, 06--09 Jul 2018.

\bibitem{BLMbook}
Stephane Boucheron, Gabor Lugosi, and Pascal Massart.
\newblock {\em {Concentration inequalities. A nonasymptotic theory of
  independence}}.
\newblock {Oxford University Press}, 2013.

\bibitem{cai-ma-2022}
Tony Cai and Rong Ma.
\newblock Theoretical foundations of t-sne for visualizing high-dimensional
  clustered data.
\newblock arXiv preprint, \url{https://arxiv.org/abs/2105.07536}, 2022.

\bibitem{book2}
Benyamin Ghojogh, Mark Crowley, Fakhri Karray, and Ali Ghodsi.
\newblock {\em Elements of Dimensionality Reduction and Manifold Learning}.
\newblock Springer, 2023.

\bibitem{engineering-example}
Parisa Hajibabaee, Farhad Pourkamali-Anaraki, and Mohammad~Amin
  Hariri-Ardebili.
\newblock An empirical evaluation of the t-sne algorithm for data visualization
  in structural engineering.
\newblock 2021.

\bibitem{hinton2002stochastic}
Geoffrey~E Hinton and Sam Roweis.
\newblock Stochastic neighbor embedding.
\newblock {\em Advances in neural information processing systems}, 15, 2002.

\bibitem{KMJ}
Eric~M. Johnson, William Kath, and Madhav Mani.
\newblock Embedr: Distinguishing signal from noise in single-cell omics data.
\newblock {\em Patterns}, 3(3):1--14, 2022.

\bibitem{biology-example2}
Dmitry Kobak and Philipp Berens.
\newblock The art of using t-sne for single-cell transcriptomics.
\newblock {\em Nature communications}, 10(1):1--14, 2019.

\bibitem{biology-example}
W.~Li, J.E. Cerise, Y.~Yang, and H.~Han.
\newblock Application of t-sne to human genetic data.
\newblock {\em Journal of bioinformatics and computational biology}, 15(4),
  2017.

\bibitem{Linderman2019}
George~C. Linderman, Manas Rachh, Jeremy~G. Hoskins, Stefan Steinerberger, and
  Yuval Kluger.
\newblock Fast interpolation-based t-sne for improved visualization of
  single-cell rna-seq data.
\newblock {\em Nature Methods}, 16(3):243–245, 2019.

\bibitem{linderman-steinerberger}
George~C. Linderman and Stefan Steinerberger.
\newblock Clustering with t-sne, provably.
\newblock {\em SIAM Journal on Mathematics of Data Science}, 1(2):313--332,
  2019.

\bibitem{SSArxiv}
Uri Shaham and Stefan Steinerberger.
\newblock Stochastic neighbor embedding separates well-separated clusters.
\newblock {\em arXiv:1702.02670}, 2017.

\bibitem{JMLR:v15:vandermaaten14a}
Laurens van~der Maaten.
\newblock Accelerating t-sne using tree-based algorithms.
\newblock {\em Journal of Machine Learning Research}, 15(93):3221--3245, 2014.

\bibitem{JMLR:v9:vandermaaten08a}
Laurens van~der Maaten and Geoffrey Hinton.
\newblock Visualizing data using t-sne.
\newblock {\em Journal of Machine Learning Research}, 9(86):2579--2605, 2008.

\bibitem{physics-example}
XueGuang Zhang, Yanqiu Feng, Huan Chen, and QiRong Yuan.
\newblock Powerful t-sne technique leading to clear separation of type-2 agn
  and h ii galaxies in bpt diagrams.
\newblock {\em The Astrophysical Journal}, 905(2), 2020.

\end{thebibliography}
    
\appendix

\section{}
\label{appendix:convergence-of-discrete-gradient}

Suppose that $\mu_X$ is sub-Gaussian, let $X_1, \dots, X_n$ be i.i.d. random variables drawn from $\mu_X$, and let $Y^* = (Y_1, \dots, Y_n)$ be the output given by t-SNE.
Let
\(
    \mu_n = \frac{1}{n} \sum_i \delta_{(X_i, Y_i)}
\)
and suppose that $\mu_n$ converges weakly to $\mu$.

Since $Y^*$ is a minimizer of $L_{n,\rho}(X,Y)$ we have
\begin{equation}\label{equation:discrete-gradient-is-zero}
    \frac{\partial L_{n,\rho}(X,Y)}{\partial y_i}
    =
    \sum_{j=1, j \neq i}^n
        (p_{ij} - q_{ij})
        \frac{Y_i - Y_j}{1 + \lVert Y_i - Y_j \Vert^2}
    = 0
\end{equation}
for all $i$.

\begin{prop}\label{prop:appendix-prop}
Suppose that $\mu_X$ has either compact support and a $C^1$ density $f(x)$ or $\mu_X$ has sub-Gaussian tails.
Then for $\mu$-a.e. $x$ and $y$ we have
\[
    \int 
        (p(x,x') - q(y, y'))
        \frac{y - y'}{1 + \lVert y - y' \rVert^2}
    \, d\mu(x', y')
    =
    0.
\]
\end{prop}

\begin{proof}
Let $\alpha(x,y)$ be a test function.
It suffices to show that
\[
    \iint   
        \alpha(x,y)
        (p(x,x') - q(y, y'))
        \frac{y - y'}{1 + \lVert y - y' \rVert^2}
    \, d\mu d\mu
    =
    0.
\]

Recalling that definitions of $p_n, q_n$, we see from (\ref{equation:discrete-gradient-is-zero}) that, for all $i$,
\[
    \frac{1}{n^2}
    \sum_{j=1, j \neq i}^n 
        (p_n(X_i, X_j) - q_n(Y_i, Y_j))
        \frac{Y_i - Y_j}{1 + \lVert Y_i - Y_j \rVert^2}
    = 0.
\]
In particular, 
\begin{align*}
    &
    \iint 
        \alpha(x,y)
        (p_n(x,x') - q_n(y, y'))
        \frac{y - y'}{1 + \lVert y - y' \rVert^2}
    \, d\mu_n d\mu_n
    \\
    = &
    \frac{1}{n^2}
    \sum_{1\leq i,j \leq n}
        \alpha(X_i, Y_i)
        (p_n(X_i, X_j) - q_n(Y_i, Y_j))
        \frac{Y_i - Y_j}{1 + \lVert Y_i - Y_j \rVert^2}
    =
    0.
\end{align*}
From Lemma \ref{lemma:lower-bound-p-comparison} and Lemma \ref{lemma:Z_q_n-convergence}, we have
\(
    p_n(X_i, X_j) - q_n(Y_i, Y_j)
    =
    p(X_i, X_j) - q(Y_i, Y_j)
    +
    o(1).
\)
Since $\alpha$ and $\frac{y - y'}{1 + \lVert y - y' \rVert^2}$ are bounded, it follows that
\begin{align*}
    &
    \iint 
        \alpha(x,y)
        (p_n(x,x') - q_n(y, y'))
        \frac{y - y'}{1 + \lVert y - y' \rVert^2}
    \, d\mu_n d\mu_n
    \\
    = &
    \iint 
        \alpha(x,y)
        (p(x,x') - q(y, y'))
        \frac{y - y'}{1 + \lVert y - y' \rVert^2}
    \, d\mu_n d\mu_n
    +
    o(1).
\end{align*}
By Lemma \ref{lemma:calP-p-convergence} and 
Lemma \ref{lemma:calP-q-convergence} below, we have
\begin{align*}
    & \iint 
        \alpha(x,y)
        (p(x,x') - q(y, y'))
        \frac{y - y'}{1 + \lVert y - y' \rVert^2}
    \, d\mu_n d\mu_n
    \\
    = &
    \iint 
        \alpha(x,y)
        (p(x,x') - q(y, y'))
        \frac{y - y'}{1 + \lVert y - y' \rVert^2}
    \, d\mu d\mu
    + o(1).
\end{align*}
So, since $\alpha$ was an arbitrary test function, 
\[
     \int 
        (p(x,x') - q(y, y'))
        \frac{y - y'}{1 + \lVert y - y' \rVert^2}
    \, d\mu(x', y')
    =
    0
\]
for $\mu$ almost every $(x,y)$ as required.
\end{proof}

\begin{lemma}\label{lemma:calP-p-convergence} The following holds as $n\to \infty$,
\[
    \iint 
        \alpha(x,y)
        p(x,x')
        \frac{y - y'}{1 + \lVert y - y' \rVert^2}
    \, d\mu_n d\mu_n
    =
    \iint 
        \alpha(x,y)
        p(x,x')
        \frac{y - y'}{1 + \lVert y - y' \rVert^2}
    \, d\mu d\mu
    +
    o(1).
\]
\end{lemma}

\begin{proof}
This follows by mimicking the proof of Lemma \ref{lemma:lower-bound-plogp-integral-convergence}.
\end{proof}

\begin{lemma}\label{lemma:calP-q-convergence} The following holds as $n\to \infty$,
\[
    \iint
        \alpha(x,y)
        q(y,y') 
        \frac{y-y'}{1 + \lVert y - y' \rVert^2}
    \, d\mu_n d\mu_n
    =
    \iint 
        \alpha(x,y)
        q(y,y') 
        \frac{y-y'}{1 + \lVert y - y' \rVert^2}
    \, d\mu d\mu
    +
    o(1).
\]
\end{lemma}

\begin{proof}
This follows by mimicking the proof of Lemma \ref{lemma:Z_q_n-convergence}.
\end{proof}

\end{document}